\documentclass[reqno,11pt,a4paper]{amsart}

\usepackage{mathtools}

\newcommand{\ds}{\displaystyle}

\newcommand{\tensor}{\otimes}

\newcommand{\op}{\mathcal}

\newcommand{\cdc}{,\dots,}

\newcommand{\mxdops}{\op{O}ps^{\mathsf{MxCpx}}}
\newcommand{\rotops}{\op{O}ps^{\mathsf{Rot}}}

\newcommand{\lhc}{\mathsf{CC}}
\newcommand{\nhc}{\mathsf{CC^-}}

\newcommand{\phc}{\mathsf{CC^\text{per}}}
\newcommand{\hlhc}{\mathsf{HC}}
\newcommand{\hnhc}{\mathsf{HC^-}}
\newcommand{\hphc}{\mathsf{HC^\text{per}}}
\newcommand{\tcc}{\mathsf{CC}^\theta}
\newcommand{\thc}{\mathsf{HC}^\theta}

\usepackage{amsmath}%
\usepackage{amsthm}
\usepackage{amsfonts}%
\usepackage{amssymb}%
\usepackage{graphicx}
\usepackage{xy,amsthm,enumerate,xypic,array}
\usepackage{epstopdf,multirow}  
\usepackage{bbm}
\usepackage{tikz-cd}
\usepackage{tkz-euclide}
\usepackage{stmaryrd}
\usepackage{tikz}
\usetikzlibrary{matrix}
\usetikzlibrary{positioning, calc}
\usetikzlibrary{shapes, fit}
\usetikzlibrary{arrows}  
\usetikzlibrary{calc,3d}
\usetikzlibrary{decorations,decorations.pathmorphing}
\usetikzlibrary{through}
\tikzset{ext/.style={circle, draw,inner sep=1pt},int/.style={circle,draw,fill,inner sep=1pt},nil/.style={inner sep=1pt}}
\tikzset{exte/.style={circle, draw,inner sep=1pt},inte/.style={circle,draw,fill,inner sep=3pt}}
\tikzset{diagram/.style={matrix of math nodes, row sep=1.5em, column sep=0.5em, text height=1.5ex, text depth=0.25ex}}
\tikzset{diagram2/.style={matrix of math nodes, row sep=0.5em, column sep=0.5em, text height=1.5ex, text depth=0.25ex}}
\input{xy}
\xyoption{all}

\numberwithin{equation}{section}

\newtheorem{theorem}{Theorem}[section]
\theoremstyle{plain}

\newtheorem*{nonumbertheoremB}{Theorem B}
\newtheorem*{nonumbertheoremA}{Theorem A}

\newtheorem{corollary}[theorem]{Corollary}
\newtheorem{lemma}[theorem]{Lemma}
\newtheorem{proposition}[theorem]{Proposition}
\newtheorem{propdef}[theorem]{Definition/Proposition}

\theoremstyle{definition}
\newtheorem{definition}[theorem]{Definition}
\newtheorem{example}[theorem]{Example}
\newtheorem{const}[theorem]{Construction}

\newtheorem{remark}[theorem]{Remark}

\allowdisplaybreaks[2]

\addtocounter{MaxMatrixCols}{2}

\newcommand{\cP}{{\mathcal{P}}}
\newcommand{\cQ}{{\mathcal{Q}}}

\newcommand{\cM}{{{\mathcal M}}}

\newcommand{\BV}{{\mathsf{BV}}}
\newcommand{\BVKGra}{{\mathsf{BVKGra}}}
\newcommand{\BVKGraphs}{{\mathsf{BVKGraphs}}}

\newcommand{\Gra}{{\mathsf{Gra}}}
\newcommand{\Graphs}{{\mathsf{Graphs}}}

\newcommand{\Ger}{{\mathsf{Ger}}}
\newcommand{\Lie}{{\mathsf{Lie}}}
\newcommand{\sLie}{{\mathfrak{s}^{-1}\mathsf{Lie}}}

\newcommand{\FM}{{\mathsf{FM}}}
\newcommand{\vKGra}{{\mathsf{vKGra}}}
\newcommand{\vKGraphs}{{\mathsf{vKGraphs}}}

\newcommand{\Tpoly}{T_{\rm poly}}
\newcommand{\Dpoly}{D_{\rm poly}}
\newcommand{\TDpoly}{\tilde{D}_{\rm poly}} 
\newcommand{\OpDpoly}{\mathcal{D}_{\rm poly}} 
\newcommand{\Rformal}{\mathbb R^d_{\rm{formal}}}

\newcommand{\Grav}{{\mathsf{Grav}}}
\newcommand{\Chains}{\mathsf{Chains}_\ast}
\newcommand{\Chain}{\mathsf{Chains}}  

\renewcommand{\H}{\mathbb H}

\DeclareMathOperator{\vol}{vol}

\DeclareMathOperator{\End}{End}
\DeclareMathOperator{\Hom}{Hom}

\DeclareMathOperator{\Tw}{\mathsf{Tw}}

\DeclareMathOperator{\Div}{div_\omega}

\newcommand{\notadp}
{{
		\begin{tikzpicture}[baseline=-.55ex,scale=.35, every loop/.style={}]
		\node[circle,draw,fill,inner sep=.5pt] (a) at (0,0) {};
		\draw (a) edge[loop] (a);
		\draw (-.2,-.2) -- (.2,.5);
		\end{tikzpicture}}}

\newcommand{\tadp}
{{
		\begin{tikzpicture}[baseline=-.75ex,scale=.8, every loop/.style={}]
		\node[circle,draw,fill,inner sep=.7pt] (a) at (0,0) {};
		\draw (a) edge[loop] (a);
		\end{tikzpicture}}}

\begin{document}

\title{Gravity Formality}
\author{Ricardo Campos} \author{Benjamin C. Ward}


\begin{abstract}
    We show that Willwacher's cyclic formality theorem can be extended to preserve natural Gravity operations on cyclic multivector fields and cyclic multidifferential operators. We express this in terms of a homotopy Gravity quasi-isomorphism with explicit local formulas. For this, we develop operadic tools related to mixed complexes and cyclic homology and prove that the operad $\mathsf M_\circlearrowleft$ of natural operations on cyclic operators is formal and hence quasi-isomorphic to the Gravity operad.
\end{abstract}

\maketitle

\tableofcontents
\section*{Introduction} 

The original deformation quantization problem aims to obtain a formal deformation of the associative product of functions of a Poisson manifold $M$, called a \textit{star product}. The space governing such deformations is essentially the Lie algebra of multidifferential operators $\Dpoly$, the smooth version of the Hochschild complex of the algebra $C^\infty(M)$. In his celebrated paper \cite{Kontsevich}, Kontsevich showed that the Lie algebra $\Dpoly$ is formal i.e., it is quasi-isomorphic to its homology, the Lie algebra of multivector fields $\Tpoly$. His proof involves the construction of the ``formality morphism'', a homotopy quasi-isomorphism of Lie algebras
$$\mathcal U \colon \Tpoly \to \Dpoly,$$
with explicit local formulas depending on integrals over configurations of points and expressed in terms of graphs. This results solves the deformation quantization problem by establishing a correspondence between formal Poisson structures and star products (bijective up to gauge equivalence). 

Kontsevich's result, however, ignores the richer structures existent on $\Tpoly$ and $\Dpoly$. Let now $M$ be an oriented $D$-dimensional manifold with a fixed volume form $\omega$. The pull back of the de Rham differential via contraction with $\omega$ endows the space $\Tpoly$ with the structure of a $\BV$ algebra.  On the other hand, 
there is a natural action of the cyclic group of order $n+1$ on $\Dpoly^n$ given by ``integration by parts'' which, after the cyclic Deligne's conjecture (see \cite{CycDel}), induces a natural $\BV_\infty$ algebra structure on $\Dpoly$. The natural question to ask is whether Kontsevich's formality morphism can be extended to a $\BV_\infty$ quasi-isomorphism. Tamarkin \cite{Tamarkin,Hinich} constructed a non-explicit $\Ger_\infty$ (homotopy Gerstenhaber) quasi-isomorphism $\Tpoly \to \Dpoly$ depending on a solution of Deligne's conjecture whose underlying $\Lie_\infty$ morphism was later shown by Willwacher \cite{homotopybracesformality} to be homotopy equivalent to Kontsevich's map if one uses the Alekseev-Torossian associator to construct a solution to Deligne's conjecture. Furthermore, Willwacher shows that the original formality morphism can be strictly extended to a $\Ger_\infty$ morphism. The full extension to the BV setting was given by the first author \cite{Campos} who constructed a $\BV_\infty$ quasi-isomorphism $\Tpoly \to \Dpoly$ with explicit local formulas depending on integrals over configurations of framed points.  One advantage of incorporating these richer structures into the discussion is that we may now view the algebraic operations as being parametrized by geometric objects, namely by the moduli spaces of genus zero surfaces with parametrized boundary components. \\

The subspace of cyclic invariants of $\Dpoly$, denoted by $\Dpoly^\sigma : = \bigoplus_{n\geq 0} (\Dpoly^n)^{\mathbb{Z}_{n+1}}$, is preserved by the Lie bracket and the Hochschild differential. The differential graded Lie algebra $\Dpoly^\sigma$ is associated to a different deformation problem, namely the construction of \textit{closed star products}. This led to the conjecture of an analogous formality statement, the ``cyclic formality conjecture'' \cite{Shoikhet}.
Let $\Div\colon \Tpoly^\bullet \to \Tpoly^{\bullet -1}$ be the divergence operator on the space of multivector fields. In \cite{WC} Willwacher gave an affirmative answer to the cyclic formality conjecture by constructing a homotopy Lie quasi-isomorphism
$$ \mathcal U^{cyc} \colon (\Tpoly[u], u\Div)\to (\Dpoly^\sigma, d_{Hoch}).$$

As in the non-cyclic case, both of these objects have structures richer than just Lie algebras.  Namely, viewing these objects as models for cyclic invariants associated to the non-cyclic case above, it will be possible to show that they each have operations parametrized by models of the moduli spaces $\op{M}_{\ast}$ of genus zero surfaces with {\it unparametrized} boundary components. 

To make this precise we use the presentation of the Gravity operad, introduced by Getzler in \cite{Geteq}.  The graded vector spaces $\Sigma H_\ast(\op{M}_{n+1})$ form an operad $\mathsf{Grav}$ which injects into $\mathsf{Ger}$, which is generated operadically by the classes of points in $H_0(\op{M}_{n+1})$ (ranging over $n\geq 2$), and whose sub-operad of top degree homology $\Sigma H_{n-2}(\op{M}_{n+1})$ is isomorphic to the suspension of the Lie operad $\mathfrak{s}\mathsf{Lie}$.  In particular every gravity algebra is a (shifted) Lie algebra.

Both spaces $(\Tpoly[u], u\Div)$ and $H(\Dpoly^\sigma)$ are naturally gravity algebras with first bracket equal to the usual Lie bracket. The natural question to ask is then whether Willwacher's homotopy Lie quasi-isomorphism can be extended to the Gravity setting, as conjectured in \cite{Ward2}. However, before attempting to answer this question one must find a $\Grav_\infty$ structure on $\Dpoly$ inducing the Gravity structure in homology, which is in some sense a dual version of the cyclic Deligne's conjecture.

In \cite{Ward2} the second author constructed the operad $\mathsf M_\circlearrowright$, a cyclic variation of the braces/minimal operad $\mathsf M$ that acts naturally on spaces of cyclic invariants such as $\Dpoly^\sigma$, and whose homology is $\mathsf{Grav}$. Our first result shows that the dg operad $\mathsf M_\circlearrowright$ is formal.

\begin{nonumbertheoremA}
	The operad $\mathsf M_\circlearrowright$ is quasi-isomorphic to $\Grav$.
\end{nonumbertheoremA}

The proof of this theorem combines three ingredients:  formality of the framed little disks after \cite{GS10}, the homology calculations of \cite{Ward2}, and the theory of cyclic homology of operads valued in mixed complexes.  This final ingredient is developed in section $\ref{secmixedcpx}$ and should be of independent interest.

From Theorem A we obtain a $\Grav_\infty$ structure on $\Dpoly$ after picking a homotopy lift $\Grav_\infty \stackrel{\sim}\to \mathsf M_\circlearrowright$. Having this $\Grav_\infty$ structure on $\Dpoly$ we can formulate the main result of this paper.

\begin{nonumbertheoremB}
Let $M$ be an oriented smooth manifold with a fixed volume form $\omega$. 
	There is a $\Grav_\infty$ quasi-isomorphism $(\Tpoly[u],u\Div)\to (\Dpoly^\sigma,d_{Hoch})$ extending Willwacher's $\mathsf{Lie}_\infty$ quasi-isomorphism.
\end{nonumbertheoremB}
In particular, the first component is the (cyclic \cite{Shoikhet}) HKR map. In the $M=\mathbb R^D$ case this formula admits an explicit expression in terms of integrals over configuration spaces in the upper half plane, parametrized by graphs, similar to the original paper from Kontsevich.

We emphasize the paradigm when considering formality-like theorems, that the natural structure on $\Dpoly^\sigma$ is not that of a $\Grav_\infty$ algebra but rather that of a $\mathsf M_\circlearrowright$ algebra, the same way that the natural structure on $\Dpoly$ is not the one of a $\Ger_\infty$/$\mathsf{BV}_\infty$ algebra but rather the Braces/Cyclic Braces structure. For this reason, operadic tools and concretely the language of operadic bimodules are a neat way to work simultaneously with the $\mathsf{M}_\circlearrowright$ algebra structure on $\Dpoly^\sigma$ and the Gravity algebra structure on $\Tpoly$.

\subsection*{Organization}  This paper is organized as follows.  We begin in section $\ref{secmixedcpx}$ by studying the interaction of operads, mixed complexes, and cyclic homology.  We then apply this theory in section $\ref{sec2}$ to prove Theorem A and to define the $\mathsf{Grav}_\infty$ structures that will be the subject of Theorem B.  In section $\ref{sec3}$ we apply our constructions from section $\ref{secmixedcpx}$ to categories of colored operads and operadic bimodules.  The resulting structures are then used in section $\ref{sec4}$ to prove Theorem B in the case $M=\mathbb{R^D}$ using the theory of operadic torsors.  Finally in section $\ref{sec5}$ we globalize the results using a suitable modification of the usual formal geometry techniques developed in \cite{Kontsevich}.

\subsection*{Notation and conventions.}  We work in the category of differential graded (dg) vector spaces over a field $k$ of characteristic $0$. We use the notation $\Sigma$ to denote the suspension of vector spaces and $\mathfrak s$ to denote operadic suspension, such that for a vector space $V$, $(\Sigma V)_d = V_{d-1}$ and $\Sigma V$ is an $\op O$ algebra if and only if $V$ is an $\mathfrak s \op O$ algebra, for any operad $\op O$.

We assume familiarity with operads, operadic twisting, and graph complexes.  A table of the graph complex operads appearing in this paper and relevant references follows:

$\bigskip$
\begin{center}
	\begin{tabular}{c|c|c|c}
		Notation & Graphs & Differential & c.f.\  \\	\hline	
		$\mathsf{B}$ & planar rooted  & none 	& e.\ g.\	\\
		&  trees &  	&  \cite{GV}	\\ \hline 
		$\mathsf{M}$ &  stable planar rooted trees w/ & via $\mathsf{Tw}$	& \cite{KS} \\ & internal and external vertices & of $\mathsf{B}$  & 	\\ \hline 
		$\mathsf{B}_\circlearrowright$ & planar connected & none 	& \cite{Ward2} 	\\
		& and genus $0$ &  	&   	\\ \hline 
		$\mathsf{M}_\circlearrowright$ & planar, connected, stable, genus 0, & via $\mathsf{Tw}$	& \cite{Ward2} \\
		&  with internal and external vertices & of $\mathsf{B}_\circlearrowright$	& \\ \hline 
		$\mathsf{Gra}$ &  graphs without & none	& \cite{homotopybracesformality} \\
		&  tadpoles & 	& \\ \hline 
		$\mathsf{Graphs}$ &  internal and external vertices & via $\mathsf{Tw}$	& \cite{homotopybracesformality} \\ & and no tadpoles & of $\mathsf{Gra}$ 	& \\ \hline 
		$\mathsf{vKGra}$ &  boundary and bulk vertices & $\partial(v)=\tadp $	& Section $\ref{sec: Gra}$ \\ & with tadpoles and powers of $v$ & 	& 
		
	\end{tabular}
\end{center}
$\bigskip$

\subsection*{Acknowledgements} We would like to thank the Max Planck Institute and the organizers of the the MPIM Program on Higher Structures in Geometry and Physics from Winter 2016, as well as Stockholm University, for fostering our collaboration on this project.  It is a pleasure to thank G.\ Horel, K.\ Irie, A.\ Khoroshkin, B.\ Vallette and T.\ Willwacher for helpful conversations.
The first author been supported by the Swiss National Science Foundation, grant 200021\_150012, and by the NCCR SwissMAP funded by the Swiss National Science Foundation.

\section{Operads in Mixed Complexes and $S^1$-operads.}\label{secmixedcpx}  In this section we consider the interaction of mixed complexes, operads, and cyclic homology.

\begin{definition} 
 A mixed complex is a triple $(V, d, \Delta)$ such that $(V, d)$ is a cochain complex and $\Delta\colon V\to V$ is a degree $-1$ operator such that $\Delta^2=0$ and $d\Delta+\Delta d =0$.
\end{definition}

 The category of mixed complexes is naturally a symmetric monoidal category with monoidal unit $(k,0,0)$.  The monoidal product is
\begin{equation*}
(A,d_A,\Delta_A)\tensor(B,d_B,\Delta_B)=(A\tensor_k B, d_A\tensor 1_B+1_A\tensor d_B,  \Delta_A\tensor 1_B+1_A\tensor \Delta_B) 
\end{equation*}
where we follow the Koszul rule for evaluation over a tensor product.  Explicitly $d(a \tensor b) = d(a) \tensor b + (-1)^{|a|}a \tensor d(b)$.

Since mixed complexes form a symmetric monoidal category, one can talk about operads valued in mixed complexes.  The category of such will be denoted $\mxdops$.  An object in $\mxdops$ is given by a triple $(\op{O}, d, \Delta)$; where $\op{O}$ is a graded operad and where $d$ and $\Delta$ are maps of $\mathbb{S}$-modules which anti-commute and which are compatible with the operad structure. 

If $(A,d,\Delta)$ is a mixed complex, the operad $End_A$ can be viewed as an operad in $\mxdops$, but it has more structure.  Thus we introduce the following definition:

\begin{definition}  An $S^1$-operad is an operad under the operad $H_\ast(S^1)$.  The category of such is denoted $S^1$-$\op{O}ps$.
\end{definition}

We denote the fundamental class of $S^1$ by $\delta$ and by abuse of notation we often use $\delta$ to denote its image in an $S^1$-operad.

\begin{const}  Let $H_\ast(S^1)\to\op{Q}$ be a morphism of dg operads.  Define $\Delta:=\{\delta,-\}$ where $\{-,-\}$ is the external Lie bracket associated to $\op{O}$ (see e.g.\ \cite{Ward2} Lemma 1.9).  Explicitly for $a\in\op{O}(n)$ of degree $d$ we define:
\begin{equation*}
\Delta(a):= \delta_\op{Q} \circ_1 a -(-1)^d\sum_{i=1}^n a\circ_i\delta_\op{Q}
\end{equation*}
Then $(\op{Q}, d_\op{Q},\Delta_\op{Q})$ is an operad in mixed complexes.  This gives a functor from $S^1$-$\op{O}ps \to \mxdops$, which we call $X$ for eXternal.
\end{const}

\begin{example}\label{Gerex}  Viewing the operad $\mathsf{Ger}$ as a suboperad of the $S^1$-operad $\mathsf{BV}$, we define $\Delta$ as above, and then show it restricts to these subspaces.  Hence, $(\mathsf{Ger}, 0, \{\delta_{\mathsf{BV}},-\})$ is an operad in mixed complexes.  Since the operator $\{\delta_{\mathsf{BV}},-\}$ captures the rotation of a configuration of little disks, we will also write $(\mathsf{Ger}, 0, R)$ for this object in $\mxdops$.
\end{example}

\begin{example}\label{graex}  More generally, we define $\mathsf{Gra}(n)$ to be the $S_n$-module spanned by graphs with $n$ numbered vertices having no tadpoles.  (Recall a tadpole is an edge which is incident to the same vertex at both ends).  Insertion of graphs makes $\mathsf{Gra}$ an operad; in particular it is a suboperad of the $S^1$-operad of all graphs in which $\delta$ is the tadpole graph (one edge and one vertex).  One may then form an operad in mixed complexes $(\mathsf{Gra}, 0, \{\delta,-\})$.  There is an inclusion of $(\mathsf{Ger}, 0, \{\delta_{\mathsf{BV}},-\})\hookrightarrow (\mathsf{Gra}, 0, \{\delta,-\})$ in $\mxdops$ given by sending the commutative product to the graph with two vertices and no edges and sending the bracket to the graph with two vertices connected by an edge.  We will revisit this example in greater detail in Section $\ref{sec: Gra}$.
\end{example}

\begin{example}  If $\op{X}$ is a operad in the category of $S^1$-spaces, its singular chains $(S_\ast(\op{X}),d,\Delta)$ are naturally an operad in mixed complexes.  Here $\Delta$ is realized at the level of singular chains via projection $\Delta^1\to S^1=\Delta^1/\partial\Delta^1$ and the shuffle product for simplicies.  We will often consider the case $\op{X}=\op{D}_2$, the little disks operad.
\end{example}

\begin{example}\label{algebras}  If $(A, d, \Delta)$ is a mixed complex then by default we consider $\End_A\in S^1$-$\op{O}ps$ by $\delta\mapsto \Delta$.  We may also view $\End$ as internal to $\mxdops$ by defining $\End_{A}^{\mathsf{mxd}} := X(\End_A)$.  The terminology ``an algebra over'' either an object in $S^1$-$\op{O}ps$ or $\mxdops$ is understood as a morphism to the respective $\End$.
\end{example}

Observe the following {\bf non-example:} the minimal operad $(\mathsf{M},d,R)$ of \cite{KS} is not an object in $\mxdops$.  Here $R$ is as defined in \cite{Ward2}; it moves the root from black to white in all ways and from white to zero.  This is a square zero operator which commutes with $d$ but it does not distribute over the compositions maps.  In order to deal with this we introduce the following weaker notion:

\begin{definition}  Let $\op{P}$ be a dg operad and $\rho$ a degree $- 1$, square zero operator on the underlying dg $\mathbb{S}$-module $\rho_n\colon \op{P}(n)\to\op{P}(n)$ (so $d\rho+\rho d=0$).  A pair $(\op{P}, \rho)$ is called a {\it rotational operad} if $\rho(a\circ_i \rho(b))= \rho(a)\circ_i \rho(b)$.  We denote the category of rotational operads as $\rotops$.  
\end{definition}

\begin{example}  Every operad in mixed complexes can be viewed as a rotational operad, via $\rho=\Delta$, (but not vice versa as per the following example).  The induced functor will be denoted $\iota\colon\mxdops\to\rotops$.
\end{example}

\begin{example}\label{minex}  $(\mathsf{M},d,R)$ is a rotational operad.  This follows from Lemma 2.6 of \cite{Ward2}.
\end{example}

\begin{example}\label{Zex}  Every $S^1$-operad may be viewed as a rotational operad 
by defining $\rho:=\delta\circ_1 -$.
\end{example}

\begin{example}  Consider the topological operad of spineless cacti $\op{C}act$ with level-wise $S^1$ action given by moving the base point (aka global zero) (c.f.\ \cite{Vor, KCacti}).  This is not an operad in the category of $S^1$ spaces; consequently the induced structure on singular chains $(S_\ast(\op{C}act),d,R)$ is not an operad in mixed complexes.  However it is a rotational operad.
\end{example}

\begin{proposition}\label{eqprop}  	There is a weak equivalence of rotational operads $(\mathsf{M},d,R)\sim (S_\ast(\op{D}_2),d,\Delta)$.
\end{proposition}
 
\begin{proof}  By weak equivalence of rotational operads we mean a zig-zag of quasi-isomorphisms of dg operads which preserves the $\rho$ operator at each stage.
	
From \cite[Lemma 7.8]{Westerland} we know there exists a zig-zag of weak equivalences of topological operads connecting $\op{C}act \stackrel{\sim} \leftarrow W(\op{D}_2) \stackrel{\sim} \rightarrow \op{D}_2$ which preserve the $S^1$ actions level-wise.  Taking singular chains we have an equivalence of rotational operads $(S_\ast(\op{C}act),d,R)\sim (S_\ast(\op{D}_2),d,\Delta)$.

We now consider the inclusion of the normalization $\op{C}act^1\stackrel{\sim}\hookrightarrow\op{C}act$ after \cite{KCacti}.  The spaces $\op{C}act^1$ are CW complexes and form an operad up to homotopy by normalizing the gluing maps in $\op{C}act$ (c.f.\ \cite{KCacti}).  Taking chains we find the following sequence of homotopy operads:
\begin{equation*}
CC_\ast(\op{C}act^1)\stackrel{\sim}\to S_\ast(\op{C}act^1)\stackrel{\sim}\rightsquigarrow S_\ast(\op{C}act)
\end{equation*}
where $\rightsquigarrow$ denotes an $\infty$-quasi-isomorphism whose first component is induced by the inclusion of spaces.  We emphasize that this sequence respects the underlying mixed complex structure at each arity.
	
From \cite{KCacti} we know that the cellular chains $CC_\ast(\op{C}act^1)$ form an honest dg operad.  Hence  the composite $CC_\ast(\op{C}act^1)\stackrel{\sim}\rightsquigarrow S_\ast(\op{C}act)$, may be realized as a map of honest dg co-operads $\mathsf{B}(CC_\ast(\op{C}act^1))\stackrel{\sim}\to \mathsf{B}(S_\ast(\op{C}act))$ and this morphism allows us to construct a zig-zag of dg operads:
\begin{equation}\label{pfseq}
CC_\ast(\op{C}act^1)\stackrel{\sim}\leftarrow
\Omega(\mathsf{B}(CC_\ast(\op{C}act^1)))\stackrel{\sim}\rightarrow
\Omega(\mathsf{B}(S_\ast(\op{C}act^1)))\stackrel{\sim}\to S_\ast(\op{C}act)
\end{equation}
The ends of this sequence are rotational operads with operator $R$ induced by the $S^1$-action on the underlying spaces.  If $\op{P}$ is a rotational operad then $\Omega(\mathsf{B}(\op{P}))$ inherits the structure of an operad in mixed complexes from the Leibniz rule and the counit of the bar-cobar adjunction
$\Omega(\mathsf{B}(\op{P}))\stackrel{\sim}\to\op{P}$ is a weak equivalence of rotational operads.  Since the original $\infty$-quasi-isomorphism was compatible with the underlying mixed-complex structure, it follows that the diagram in line $\ref{pfseq}$ constitutes a weak equivalence of rotational operads between $(CC_\ast(\op{C}act^1),d,R)$ and $(S_\ast(\op{C}act),d,R)$.

To finish the proof we recall (see \cite{Ward2} Lemma 4.6) that contracting associahedra in the minimal operad commutes with the operator $R$ and gives us a weak equivalences of rotational operads $(\mathsf{M},d,R)\sim (CC_\ast(\op{C}act^1),d,R)$. \end{proof}

\begin{const}\label{twistconst}  Define a functor $\theta\colon\rotops\to \mxdops$ by taking a dg rotational operad $\op{O}$ to $(\theta_\rho(\op{O}),d,\rho)\in \mxdops$ where $\theta_\rho(\op{O})(n):=\Sigma^{-1} \op{O}(n)$, with ``twist gluings'' $a \tilde{\circ}_i b := a\circ_i \rho(b)$.  (It is easy to check that the twist gluings satisfy associativity and are compatible with $d$ and $\rho$).  For every such $\op{O}$ there is a morphism of rotational operads $\iota(\theta_\rho(\op{O}))\to\op{O}$ given by $a\mapsto \rho(a)$.  We denote the induced natural transformation $\iota\circ\theta\Rightarrow id_{\rotops}$ by $\theta^{-1}$.
\end{const}

\begin{remark}  The operad $\theta(\op{O})$ does not come with a unit for the composition in $\theta(\op{O})(1)=\Sigma{\op{O}}(1)$.  Thus here we are considering non-unital or ``pseudo-operads'' in the parlance of some authors.
\end{remark}

\begin{lemma}  Given $(\op{P},d,\rho)\in\rotops$, the natural transformation $\theta^{-1}$ above factors as:
	\begin{equation*}
	\theta(\op{P})\stackrel{\rho}\to \mbox{Im}(\rho)\hookrightarrow \ker(\rho)\to \op{P}
	\end{equation*}
\end{lemma}

We conclude this subsection by recalling, in the present terminology, a definition from \cite{Ward2}:

\begin{definition}\label{Mcircdef} Considering the rotational operad $(\mathsf{M},d,R)$, we define the dg operad $(\mathsf{M}_\circlearrowright, d)$ to be $(im(R),d)$.  In particular there exists an inclusion of dg operads $\mathsf{M}_\circlearrowright\hookrightarrow \mathsf{M}$.
\end{definition}

\subsection{Adjoints and algebras}

We have seen that if $(A,d,\Delta)$ is a mixed complex then $\End_A$ is an $S^1$-operad.  As such, algebras over operads in $\mxdops$ are controlled by morphisms to $X(\End_A)$; this prompts us to construct the left adjoint to $X$.

\begin{const}  Define a functor $W\colon\mxdops\to S^1$-$\op{O}ps$  by 
\begin{equation*}
W(\op{Q},d_\op{Q}, R)=(\op{Q}\star k[\delta])/\langle R-\{\delta,-\}\rangle
\end{equation*}
In words: take the free $S^1$-operad on the underlying operad and identify the two candidates for rotation; the original $R$ and the external bracket with the newly added $\delta$.  This $S^1$ operad is given the differential induced by $d_\op{Q}$ and the relation $d(\delta)=0$.
\end{const}

\begin{lemma}  $(W,X)$ are an adjoint pair.
\end{lemma}
\begin{proof}  Given $\phi\in Hom_{\mxdops}(A, X(B))$, we may forget the mixed complex structures and take the adjoint to forgetting the morphism from $H_\ast(S^1)$ to get a map $Free_{S^1}(A)\to B$, which we call $\tilde{\phi}$.  We then calculate
\begin{equation*}
\tilde{\phi}(R(a)-\{\Delta_{W(A)},a\})=\phi(R(a))-\tilde{\phi}(\{\Delta_{W(A)},a\})=\phi(R(a))-\{\Delta_B,\phi(a)\}
\end{equation*}
but we now remember that $\phi$ was a map of mixed complexes so this last expression equals $0$.  Thus $\tilde{\phi}$ lifts over the quotient of such expressions, that is $\tilde{\phi}\in Hom_{S^1\text{-}Ops}(W(A), B)$; and conversely.
\end{proof}
Recall that for $\op{O}\in\mxdops$, the structure of an $\op{O}$-algebra on a mixed complex $(A,d,\Delta)$ is a morphism $\op{O}\to \End_{A}^{\mathsf{mxd}} := X(\End_A)$.  Thus we immediately see:
\begin{corollary}\label{algcor}  Let $\op{O}\in\mxdops$.  The $\op{O}$-algebra structures on a mixed complex $(A,d,\Delta_A)$ are in bijective correspondence with morphisms $W(\op{O})\to \End_A$ in $S^1$-$\op{O}ps$.

\end{corollary}

\begin{example}\label{Wgerex}  $W(\mathsf{Ger})=\mathsf{BV}$.  To see this, notice $R(\mu)=\{\Delta,\mu\}=b$ (the bracket) and $R(b)=\{\Delta, b\}=0$.  In otherwords, a $W(\mathsf{Ger})$ algebra is a Gerstenhaber algebra and a mixed complex such that $\Delta$ is a derivation of the bracket, and the failure to be a derivation of the product is the bracket.  In particular a mixed complex is a (dg) BV algebra iff and only if it is a Gerstenhaber algebra for which the two inherent notions of rotation coincide.
\end{example}

\begin{remark}  The functor $ S^1$-$\op{O}ps\to\rotops$ defined in Example $\ref{Zex}$ also has a left adjoint by a similar construction, and hence we may also encode algebras over rotational operads in the category of $S^1$-operads.  However, this will not be needed for our present purposes.
\end{remark}

Let us now gather together the relevant constructions of this subsection:

\begin{tabular}{lcr}
	$\begin{aligned}[top] 
	\xymatrix{  \ar@/_/[rr]_W \ar@/_/[d]_\iota \mxdops  &  & 
		S^1\text{-}\op{O}ps  \ar@/_/[ll]_X \\
		\ar@/_/[u]_\theta \rotops  & & }
	\end{aligned}$ &
	&
	$\begin{aligned} \\ \left\{\begin{array}{l}
	\iota \text{ via inclusion}  \\ 
	X \text{ via } \Delta:=\{\delta,-\} \\
	\theta\text{ via $\Sigma$ and twist gluings } a \tilde{\circ}_i b := a\circ_i \rho(b) \\ 
	(W,X) \text{ an adjoint pair.} \end{array}	\right.\end{aligned}$  \\ 
\end{tabular}

\subsection{Levelwise cyclic homology.}

Given an operad $(\op{O},d,\Delta)\in\mxdops$ we may take the cyclic homology of each level/arity.  These spaces still form a dg operad.  We will also need to consider negative and periodic variants.  Having fixed cohomological conventions for our mixed complexes, we have $|d|=1$, $|\Delta|=-1$, $|u|=2$; we also define $v:=u^{-1}$ so that $|v|=-2$.

\begin{const}\label{lcc}  Define functors $\lhc,\nhc,\phc\colon \mxdops\to \mxdops$ by:
\begin{align*}
&\lhc(\op{O},d, \Delta)(n)= (\op{O}(n)\tensor k[v], d+\Delta u, \Delta) \\
&\nhc(\op{O},d, \Delta)(n)= (\op{O}(n)\tensor k[u], d+\Delta u, \Delta) \\
&\phc(\op{O},d, \Delta)(n)= (\op{O}(n)\tensor k[u,v], d+\Delta u, \Delta)
\end{align*}
with the operad structure:
\begin{equation*}
(a\tensor v^r)\circ_i (b \tensor v^s) :=(a\circ_i b)\tensor v^{r+s}
\end{equation*}
It is then straight forward to check associativity and compatibility of the differential and the operad structure.

The functor $\lhc$ will be called the level-wise cyclic chain functor and its homology is called the level-wise cyclic homology, denoted $\hlhc(\op{O})$.  We similarly refer to the negative $\nhc$ and $\phc$ periodic variants.  Notice that we call this constructions cyclic {\it homology} regardless of the degree conventions of our mixed complexes.  This is because we are considering the mixed complexes themselves and not functions on them.  We also observe that there is a useful modification of this construction which takes the completed tensor product, but since we will be considering $\op{O}$ which are bounded and of finite type, we are not concerned with this distinction. 

A weak equivalence in the category $\mxdops$ is a zig-zag of morphisms each of which are level-wise quasi-isomorphisms.  Note that since $\lhc,\nhc,\phc$ preserve level-wise quasi-isomorphisms, they preserve weak equivalences. 
\end{const}

\begin{definition}\label{tccdef}  We define the functor $\tcc\colon \rotops\to \mxdops$ by $\tcc:=\lhc\circ \theta$.  If $\op{O}\in\mxdops$ we write $\tcc(\op{O})$ in place of $\tcc(\iota(\op{O}))$ without further ado.  We also define $\thc(-)\coloneqq H^\ast(\tcc(-))$.
\end{definition}

Spelling out the definition of the functor $\tcc$, we see that as an $\mathbb S$-module we can identify $\tcc (\op O) = \Sigma^{-1} \op O [v]$ and under this identification the composition maps are given by ``twisted gluings'' 
$$(p\tensor v^r) \tilde{\circ}_i (q\tensor v^s) = (p \circ_i \rho(q))\tensor v^{r+s}, \ \text{ for } p,q\in \op O. $$

Given $\op{O}\in\mxdops$, there is a short exact sequence in $\text{dg-}\mathbb{S}$-Mod
\begin{equation}\label{ses}
0\to\nhc(\op{O})\hookrightarrow  \phc(\op{O})\stackrel{u}\longrightarrow \Sigma^{-2} \lhc(\op{O})\to 0
\end{equation}
the map labeled by $u$ is ``multiplication by $u$'' and sends $v$ to $1$, $1$ to $0$, etc. 

The connecting homomorphism in the associated long exact sequence can be described via $\theta$ (Construction $\ref{twistconst}$).  First observe that there is an isomorphism of $\mathbb{S}$-modules $\thc(\op{O})\cong \Sigma^{-1}\hlhc(\op{O})$ and this endows the right hand side with an operad structure.  

\begin{lemma}\label{morlem}
The boundary map in the long exact sequence associated to equation $\ref{ses}$ is a morphism of operads $\thc(\op{O})\cong\Sigma^{-1}\hlhc(\op{O})\to \hnhc(\op{O})$.  
\end{lemma}
\begin{proof}
This follows from the fact that if $c_0+c_1u^{-1}+\dots$ is a $d+u\Delta$ cycle in $\lhc(\op{O})(n)$ then the image of its homology class under the connecting homomorphism is $[\Delta(c_0)]$.
\end{proof}

  We also remark that if $d_\op{O}=0$, the connecting homomorphism coincides with the homology of $\theta^{-1}$, else it is a combination of $\theta^{-1}$ and projection $u\mapsto 0$.

\begin{lemma}\label{morlem2}  Let $(\op{O},d,\Delta)\in\mxdops$ and suppose that $[\Delta]$ is exact on $H(\op{O},d)$ and that each $(\op{O}(n),d)$ is bounded above.  Then the morphism of operads in Lemma $\ref{morlem}$ is an isomorphism.
\end{lemma}
\begin{proof}  It is enough to show that $\hphc(\op{O})$ vanishes.  Consider a filtration of $\op{O}(r)[u,v]$ by the powers of $u$.  The exactness of $[\Delta]$ will result in the $E^2$ page of the associated spectral sequence being exactly $0$.  Since $\op{O}(r)[u,v]$ is bounded in each filtration degree, this spectral sequences converges to $\hphc(\op{O})(r)$, hence the claim.
\end{proof}

\begin{corollary}\label{wecor}
Let $(\op{O},d,\Delta)\in\mxdops$. There are maps of dg operads:
\begin{equation}\label{seq}
\tcc(\op{O})\longrightarrow (\ker(\Delta),d)\longrightarrow \nhc(\op{O})
\end{equation}
which are both weak equivalences if the conditions of Lemma $\ref{morlem2}$ are satisfied.
\end{corollary}
\begin{proof}  Define the left hand map by  $c_0+c_1u^{-1}+\dots+c_nu^{-n}\mapsto \Delta(c_0)$ in each arity.  Define the right hand map by inclusion at $u^0$ in each arity.  It is straight forward to check that these are dg operad maps.

Now we assume the conditions of Lemma $\ref{morlem}$ which implies that the composition of these two maps is a weak equivalence.  We then claim the left hand map is surjective on homology at each level.  For if $[a]$ is a class in $H(\ker(\Delta),d)$ then $[\Delta(a)]=0$ implies $[a]\in Im([\Delta])$ and hence there exists $b\in\op{O}(r)$ with $db=0$ such that $[\Delta(b)]=[a]\in H(\op{O}(r))$.  Since $\Delta(b)$ is in the image of the left hand map, the claim follows.

So if we consider the sequence on line $\eqref{seq}$, the composite being a level-wise isomorphism on homology forces the left hand map to be a level-wise injection on homology. Since it is also a level-wise surjection on homology, the left hand map is a weak equivalence.  Hence the right hand map is an weak equivalence by the 2-out-of-3 property.
\end{proof}

By a {\it truncated operad} we refer to the truncation of an operad to its arity $\geq 2$ terms.

\begin{example}\label{exger}  Consider $(\mathsf{Ger},0,R)$ as a truncated operad in mixed complexes.  Then $R$ is exact on $\mathsf{Ger}$, see \cite{Geteq}.  In arity $\geq 2$, $\nhc(\mathsf{Ger})$ has cycles and boundaries:
	\begin{equation*}
	Z(\nhc(\mathsf{Ger}))= \ker(R)\tensor k[u] \ \ \ \ \ \ \text{ and } \ \ \ \ \ \ \ B(\nhc(\mathsf{Ger}))= Im(R)\tensor uk[u]
	\end{equation*}
	so $\hnhc(\mathsf{Ger})\cong Im(R)\cong \ker(R)$.  On the other hand $\lhc(\mathsf{Ger})$ has cycles and boundaries:
	\begin{equation*}
	Z(\lhc(\mathsf{Ger}))= \mathsf{Ger}\oplus (\ker(R)\tensor vk[v]) \ \ \ \ \ \ \text{ and } \ \ \ \ \ \ \ B(\lhc(\mathsf{Ger}))=  Im(R)\tensor k[v]
	\end{equation*}
	so $\Sigma \hlhc(\mathsf{Ger})\cong \Sigma \mathsf{Ger}/Im(R)$.  In particular the generators are the $n$-fold commutative products.  The corollary gives us weak equivalences of dg operads:
		\begin{equation*}
		\tcc(\mathsf{Ger})\stackrel{\sim}\longrightarrow \mathsf{Grav}\stackrel{\sim}\longrightarrow \nhc(\mathsf{Ger})
		\end{equation*}
where $\mathsf{Grav}$ is by definition the graded operad $(\ker(R),0)$ which we call the gravity operad after \cite{Geteq}.  This weak equivalence $\tcc(\mathsf{Ger})\stackrel{\sim}\to \nhc(\mathsf{Ger})$ can be interpreted as a dg version of \cite[Corollary 2.8]{Westerland}.
\end{example}

Recall that two objects in $\mxdops$ (resp.\ $\rotops$) are said to be weakly equivalent (denoted $\sim$) if they are connected by a zig-zag of levelwise quasi-isomorphisms of dg operads which preserve the rotation operator.

From the level-wise homotopy invariance of $\tcc$ and $\nhc$ we immediately see:

\begin{corollary}\label{WEtogravcor}  If $(\op{O},d,\Delta)$ in $\mxdops$ is weakly equivalent to $(\mathsf{Ger},0,R)$, then $\tcc(\op{O})\sim \mathsf{Grav}\sim\nhc(\op{O})$ are weakly equivalent dg operads.

If $(\op{O},d,\Delta)$ in $\rotops$ is weakly equivalent to $ (\mathsf{Ger},0,R)$ (viewed as a rotational operad) then $\tcc(\op{O})\sim \mathsf{Grav}$ are weakly equivalent dg operads.
\end{corollary}

\subsection{Operations on cyclic homology.}

In this section we fix a mixed complex $(A,d_A,\delta_A)$ and consider its cyclic homology as well as negative and periodic variants.  This is the same construction as Construction $\ref{lcc}$ above, except the input and output is just a mixed complex (as opposed to an operad in mixed complexes).  Let us use the same notation to denote these constructions for both algebras and operads; so explicitly we consider chain complexes
 $\lhc(A):= (A\tensor k[v], d+\delta u)$,
 $\nhc(A):= (A\tensor k[u], d+\delta u)$, and
 $\phc(A):= (A\tensor k[u,v], d+\delta u)$.

Recall that for our mixed complex $A$ we may consider $End_A$ as an $S^1$-operad or as an operad in mixed complexes $End_A^{\mathsf{mxd}}$, via $\Delta=\{\delta,-\}$.  In this section we take the latter consideration as the default.  The following lemma will allow us to study operations on cyclic cohomology:

\begin{lemma}  Let $A=(A,d,\Delta)$ be a mixed complex.
	There is an inclusion $\nhc(End_A)\hookrightarrow End_{\nhc(A)}$ in $\mxdops$.  
\end{lemma}
\begin{proof}  Define a map $\psi_n$:
\begin{equation*}
Hom(A^{\tensor n},A)\tensor k[u]\stackrel{\psi_n}\to Hom(A[u]^{\tensor n},A[u])
\end{equation*}	
as the $k$-linear extension of the assignment:
\begin{equation*}
f\tensor u^r \mapsto \left[ (a_1u^{i_1}\tensor\dots\tensor a_nu^{i_n})\mapsto f(a_1\cdc a_n)u^{r+\sum_j u_{i_j}}\right]
\end{equation*}	
This map is clearly injective.  In particular, a multi-linear operation on $A[u]$ is in the image of this map if and only if it is $u$-linear and has bounded support in the codomain. We remark that the extension of this map to $Hom(A^{\tensor n},A)\hat{\tensor} k[u]$ would encompass all multi-linear operations in its image, but this intermediary operad is not needed for our purposes.  We now claim that the $\psi_n$ constitute a map of dg operads. 

Let us first check the differential.  The operad $End_A$ has an internal differential, call it $\partial$, induced by $d_A\in End_A(1)$.  Notice that it can be described via the operadic Lie bracket as $\partial(f)=\{d_A,f\}$.  Therefore the total differential on $\nhc(End_A)$ which is {\it a priori} of the form $\partial+\Delta u$, can be rewritten as $\{d_A,- \}+\{\delta, -\}u=\{d_A+u\delta,- \}$.  On the other hand, the operad $End_{\nhc(A)}$ has differential induced from the complex $(\nhc(A),d_A+u\delta)$ via the operadic Lie bracket.  Thus we again find $\{d_A+u\delta,- \}$, and so the differentials agree.

It is then an easy exercise to see that $\psi$ respects the operad compositions.  In particular let $f$ and $g$ be multi-linear operations on $A$ of arities $n$ and $m$.  Then we see that both $\psi(f\tensor u^r)\circ_l \psi(g\tensor u^s)$ and $\psi(f\tensor u^r\circ_l g\tensor u^s):= \psi(f\circ_l g\tensor u^{r+s})$ are evaluated at a pure tensor $\tensor_{j=1}^{n+m-1}a_ju^{i_j}$ by evaluating $f\circ_l g$ at $\tensor_j a_j$ and multiplying by $u$ to the power 
$(s+\sum_{j=l}^{l+m-1} i_j)+(r+\sum_{j=1}^{l-1} i_j+\sum_{j=l+m}^{n+m-1} i_j)$
in the former case and $r+s+\sum_j i_j$ in the latter; and these two expressions are equal. \end{proof}

\begin{remark}
	We have given the statement of the Lemma using the negative variant of cyclic cohomology because it will be the result we need subsequently.  However the same result can be proven for the other variants.
\end{remark}

\begin{corollary}\label{algcor2}  If $(A,d,\Delta)$ is an algebra over the $S^1$-operad $W(\op{O})$ then $\nhc(A)$ inherits the structure of an algebra over $\nhc(\op{O})$.
\end{corollary}
\begin{proof}
	Associated to the map of $S^1$-operads $W(\op{O})\to End_A$ is the adjoint map $\op{O}\to End_A$  in $\mxdops$ (suppressing the notation $X$ used above).  Taking $\nhc$ of this map and applying the lemma we have $\nhc(\op{O})\to \nhc(End_A)\hookrightarrow End_{\nhc(A)}$.
\end{proof}

\begin{example}\label{gravfrombvex}  If $A$ is a $\mathsf{BV}$-algebra, then $\nhc(A)$ inherits the structure of a gravity algebra via the sequence
	\begin{equation}
	\mathsf{Grav} \stackrel{\sim}\to \nhc(\mathsf{Ger})\hookrightarrow End_{\nhc(A)}
	\end{equation}	
	after Example $\ref{Wgerex}$.  
	
	More generally, combining this example with Example $\ref{exger}$ above we see that if $A$ is a $\mathsf{BV}$-algebra, there is a sequence of (truncated) dg operads:
	\begin{equation}
	\mathsf{Grav}_\infty \stackrel{\sim}\to \tcc(\mathsf{Ger})\stackrel{\sim}\to \nhc(\mathsf{Ger})\hookrightarrow End_{\nhc(A)}
	\end{equation}
	
\end{example}

We will use this construction in the following section to associate a gravity algebra to the poly-vector fields of an oriented manifold.

\section{Formality, cyclic formality, and gravity structures.}\label{sec2}  In this section we recall the statement of Kontsevich's formality theorem \cite{Kont99, Kontsevich} and the cyclic variant of the theorem due to Willwacher \cite{WC}.  We also apply our work from Section $\ref{secmixedcpx}$ to establish the $\mathsf{Grav}_\infty$ structures on the respective sides of the cyclic formality theorem that will be the subject of our results in subsequent sections.
  
    In this section we fix an oriented manifold $M$ of dimension $d$ and equip it with a fixed volume form $\omega$.  In this section and beyond we take our ground field to be the real numbers.

\subsection{Multivector Fields}

The graded vector space $\Tpoly(M)$, or just $\Tpoly$, of multivector fields on $M$ is $$\Tpoly^\bullet = \Gamma(M, {\bigwedge}^\bullet T_M),$$
where $T_M$ is the tangent bundle of $M$.  This space is naturally a $\sLie$ algebra with Lie bracket given by the Schouten-Nijenhuis bracket.  It is also a graded commutative algebra under the exterior product and these operations combine to make $\Tpoly$ a Gerstenhaber algebra. 

We now define a map $f\colon \Tpoly^\bullet(M) \to \Omega^{d-\bullet}_{dR}(M)$ that sends a multivector field to its contraction with the volume form of $M$.  This map is easily checked to be an isomorphism of vector spaces. We define the divergence operator $\Div$ to be the pullback of the de Rham differential via $f$, i.e. $\Div:=f^{-1}\circ d_{dR}\circ f$.  The square zero operator $\Div$ combines with the Gerstenhaber structure to make $\Tpoly$ a BV algebra.  

Our work in Section $\ref{secmixedcpx}$, namely Example $\ref{gravfrombvex}$, assigns to the complex $(\Tpoly[u],u\Div)$ the structure of a dg gravity algebra.  Explicit formulas for this structure can be given as the $u$-linear extension of those given in \cite[Lemma 4.4]{Geteq}.  In particular this complex is a dg $\sLie$ algebra whose bracket is the $u$-linear extension, $[Xu^k,Yu^l] := [X,Y]u^{k+l}$.

\subsection{Multidifferential Operators}
In this section we describe the differential graded Lie algebra of multidifferential operators of $M$, denoted by $\Dpoly(M)$ or just $\Dpoly$. We do an operadic construction which is less standard but allows us to introduce notation that suits our needs better.

Consider the endomorphism operad $\End(C_c^\infty(M)) =\Hom(C^\infty_c(M)^{\otimes \bullet}, C_c^\infty(M))$ on the algebra of compactly supported smooth functions on $M$, concentrated in degree zero.  We define $\OpDpoly\subset \End(C_c^\infty(M))$ to be the suboperad given by endomorphisms that vanish on constant functions\footnote{Some authors call this space the ``normalized cochains'' or ``normalized multidifferential operators'' due to this condition.} and that can be locally expressed in the form 
\begin{equation*}
\sum f\frac{\partial}{\partial x_{I_1}}\otimes \dots\otimes \frac{\partial}{\partial x_{I_n}}
\end{equation*}
where the $I_j$ are finite sequences of indices between $1$ and $dim(M)$ and $\partial/\partial x_{I_j}$ is the multi-index notation representing the composition of partial derivatives.

Associated to $\OpDpoly$ is the graded vector space $\TDpoly = \bigoplus_n\Sigma\mathfrak{s}\OpDpoly(n)$ (graded internally so that arity $n$ operators are of concentrated in degree $n$) which inherits a natural graded $\mathfrak{s}^{-1}\Lie$ algebra structure from the symmetrization of the total composition maps
\begin{equation*}
D \circ D' = \sum_{i=1}^{|D|} (-1)^{(i-1)(|D^\prime|-1)}D \circ_i D'.
\end{equation*}

The product $\mu\in \OpDpoly(2)$ of compactly supported functions is associative. This can be rewritten as $[\mu,\mu]=0$ which amounts to saying that $\mu$ is a Maurer-Cartan element of the shifted Lie algebra $\TDpoly$.  The differential graded Lie algebra of multidifferential operators $\Dpoly$ is defined to be $\TDpoly^\mu:= (\TDpoly, [\mu,-])$, the twist of $\TDpoly$ by the Maurer-Cartan element $\mu$. \\

The Hochschild-Kostant-Rosenberg map is a quasi-isomorphism of cochain complexes $\Tpoly \to \Dpoly$ that is not compatible with the Lie algebra structures. Kontsevich's result states that the obstructions of the Hochschild-Kostant-Rosenberg map to commute with the Lie algebra structure are homotopically trivial.

\begin{theorem}[Kontsevich Formality]
	There exists a homotopy $\mathfrak{s}^{-1}\Lie_\infty$ quasi-isomorphism
	$$\mathcal U\colon \Tpoly\to \Dpoly$$
	extending the Hochschild-Kostant-Rosenberg map.
\end{theorem}

We will now describe an action of the group $\mathbb Z_{n+1} = \langle\sigma_n | \sigma_n^{n+1} = e \rangle$ on $\OpDpoly(n)$.  Consider the natural map
$$\Hom(C^\infty_c (M)^{\otimes n} , C^\infty_c (M)) \to \Hom(C^\infty_c (M)^{\otimes n+1} , \mathbb R)$$
induced from the pairing $\int_M \colon C^\infty_c (M) \otimes C^\infty_c (M) \to \mathbb R$. The restriction of this map to $\OpDpoly(n)$ is injective and therefore $\OpDpoly(n)$ inherits the $\mathbb Z_{n+1}$ action of $\Hom(C^\infty_c (M)^{\otimes n+1} , \mathbb R)$ coming from cyclic permutation of the inputs. 

We define the cyclic multi-differential operators $\Dpoly^\sigma$ to be those multi-differential operators which are left invariant under the cyclic action.

\begin{proposition}\label{dgcyc}  The Lie algebra structure and differential of $\Dpoly$ restrict to this subspace making $\Dpoly^\sigma$ a differential graded Lie algebra.
\end{proposition}

Similarly to the non-cyclic setting, there is a cyclic Hochschild-Kostant-Rosenberg map \cite{Shoikhet} $\Tpoly[u]\to \Dpoly^\sigma$ which is a quasi-isomorphism of chain complexes. Willwacher's result shows that also in the cyclic setting this map can be made compatible with the Lie bracket up to homotopy.

\begin{theorem}[Willwacher's Cyclic Formality]
	There exists a homotopy $\mathfrak{s}^{-1}\Lie_\infty$ quasi-isomorphism
	$$\mathcal U^{\rm cyc} \colon \Tpoly[u]\to \Dpoly^\sigma$$
	extending the cyclic Hochschild-Kostant-Rosenberg map.
\end{theorem}

\begin{remark}\label{dgcycrmk}  The cyclic group actions combine with the operad structure to endow $\OpDpoly$ with the structure of a cyclic operad (c.f.\ \cite{GeK1}).  Proposition $\ref{dgcyc}$ actually holds for any cyclic operad  (see e.g.\ \cite{Ward2}).  We term the associated dg Lie algebra $(\op{O}^\sigma, d_\mu)$ the cyclic deformation complex of $(\op{O},\mu)$.	 The fact that $\OpDpoly$ is a cyclic operad equipped with the invariant Maurer-Cartan element $\mu$ allows us to apply \cite[Theorem C]{Ward2} to show that the dg Lie algebra structure on $\Dpoly^\sigma$ lifts to an action of the dg operad $\mathsf{M}_\circlearrowright$ of Definition $\ref{Mcircdef}$.  This will be used below to endow $\Dpoly^\sigma$ with the structure of a $\mathsf{Grav}_\infty$-algebra.  
\end{remark}

\subsection{$\mathsf{FM}_2$ and formality of $\mathsf{M}_\circlearrowright$}\label{sec: formality of little discs}
\label{sec: formality of FM2}

Let $$Conf_n(\mathbb C) = \{(x_1,\dots,x_n) \in (\mathbb C)^n | x_i \ne x_j \text{ for } i\ne j\}/\mathbb R_+ \ltimes \mathbb C$$ be the configuration space of $n$ labeled points in $\mathbb C$ modulo the action of the Lie group $\mathbb R_+ \ltimes \mathbb C$ acting by scaling and translations. Notice that $Conf_n(\mathbb C)$ is a $2n-3$ dimensional smooth manifold.

The Fulton-MacPherson topological operad $\FM_2$, introduced by Getzler and Jones \cite{Getzler-Jones} after \cite{FM} is constructed in such a way that the $n$-ary space $\FM_2(n)$ is a compactification of the $Conf_n(\mathbb C)$. The spaces $\FM_2(n)$ are manifolds with corners with each boundary stratum representing a set of points that got infinitely close.

Formally, the compactification is done by considering the closure of $Conf_n(\mathbb C)$ under the embedding $Conf_n(\mathbb C) \hookrightarrow (S^1)^{n(n-1)} \times [0,+\infty]^{n^2(n-1)^2}$ that maps every pair of points to their angle and every triple of points to their relative distances.

The first few terms are
\begin{itemize}
	\item $\FM_2(0) = \emptyset$,
	\item $\FM_2(1) = \{*\}$,
	\item $\FM_2(2) = S^1$.
\end{itemize}
The operadic composition $\circ_i$ is given by inserting a configuration at the boundary stratum at the point labeled by $i$.
For details on this construction see also \cite[Part IV]{FM} or \cite{Kont99}.

The homology of the Fulton-MacPherson operad is the Gerstenhaber operad $\Ger$ \cite{Arnold}.  The formality of this operad was established by Kontsevich with the exhibition of the following explicit zig-zag of quasi-isomorphisms.
\begin{theorem}[\cite{Kont99,discs}]
	There is a zig-zag of quasi-isomorphisms of operads
	$$\Chains(\FM_2) \to \Graphs \leftarrow \Ger.$$
\end{theorem}

The operad $\Graphs$ is an operad of graphs, constructed by considering the twisted operad of $\Gra$ (see section \ref{sec: Gra} for the definitions). Concretely, it is the suboperad of $\Tw \Gra$ consisting of graphs containing no connected components without external vertices and all internal vertices have valence at least 3. The construction of this operad using operadic twisting was first done in \cite{homotopybracesformality}.

We mention the technical point that the various projection maps $\FM_2(n+k)\to\FM_2(n)$ obtained by forgetting $k$ points of the configuration are not smooth fiber bundles. Since Kontsevich's construction requires integration of forms along fibers, one has to work in a semi-algebraic setting. In particular, the functor $\Chains$, used by Kontsevich is the functor of semi-algebraic chains (see \cite{HLTV} for an extensive study of this functor) and the morphism $\Chains(\FM_2) \to \Graphs$ is best constructed in the dual setting, as a map of cooperads
$$\omega_\bullet \colon \Graphs^* \to \Omega(\FM_2),$$
where $\Omega$ represents the functor of PA (piecewise-algebraic) forms.

\begin{remark}
	The functor $\Omega$ is not comonoidal since the canonical map $\Omega(A)\otimes \Omega(B) \to \Omega(A\times B)$ goes ``in the wrong direction", therefore $\Omega(\FM_2)$ is not a cooperad but still satisfies cooperad-like relations (see \cite{discs}). Nevertheless, by abuse of language throughout this paper we will refer to these spaces as cooperads and refer to maps such as $\Gra^*  \to \Omega(\FM_2)$  as maps of (colored) cooperads if they satisfy a compatibility relation such as commutativity of the following diagram:
	$$  \begin{tikzcd}
	\Graphs^*(n) \ar{r} \arrow{dd} & \Omega(\FM_2(n)) \arrow{d} \\
	& \Omega(\FM_2(n-k+1)\times \FM_2(k)) \\
	\Graphs^*(n-k+1)\otimes {} \Graphs^*(k) \ar{r} & \Omega(\FM_2(n-k+1))\otimes \Omega(\FM_2(k)). \ar{u}
	\end{tikzcd}$$
\end{remark}

To describe the map $\omega_\bullet$, first let us take $\Gamma$, a graph in $\Graphs^*(n)$ with no internal vertices. We define $$\omega_\Gamma \coloneqq \bigwedge_{(i,j) \text{ edge of }\Gamma} d\phi_{i,j} \in \Omega(\FM_2(n)),$$ 
where $d\phi_{i,j} = p_{i,j}^*(\vol_{S^1})$\footnote{Notice that $d\phi_{i,j}$ is not an exact form, since the angle $\phi_{i,j}$ is only well defined up to a constant.} is the pullback of the volume form of the circle via the projection map $p_{i,j}\colon \Omega(\FM_2(n)) \to \Omega(\FM_2(2)) = \Omega(S^{1})$.

If the graph $\Gamma \in \Graphs^*(n)$ contains $k$ internal vertices, one can construct a graph $\Gamma'\in \Graphs^*(n+k)$ by replacing all internal vertices by external vertices labeled in some way from $n+1$ to $n+k$. The map $\omega_\bullet \colon \Graphs^*(n) \to \Omega(\FM_2(n))$ is defined by sending $\Gamma$ to $\int_{n+1,\dots,n+k}\omega_{\Gamma'}$, where the integral runs over all possible configuration of the points that correspond to the internal vertices. 

Notice that the induced composition map $\Gra^* \to \Graphs^* \to \Omega(\FM_2)$ is just the map of commutative algebras defined by sending the edge connecting vertices $i$ and $j$ to $\phi_{i,j}$.

\begin{remark}
	The operad $\FM_2$ can be directly related to a shifted version of the homotopy Lie operad via the operad morphism
	
	\begin{equation*}
	\sLie_\infty \to \Chains(\FM_2),
	\end{equation*}
	given by sending the generator $\mu_n \in \sLie_\infty$ to the fundamental chain of $\FM_2(n)$\footnote{Note that due to our cohomological conventions the generator $\mu_n\in \sLie_\infty$ has degree $(1-n)+(2-n)=3-2n$ as desired.}. This is essentially lifting the formality zig-zag for $\sLie_\infty\subset \Ger_\infty$.
\end{remark}

The action of $S^1$ on $\FM_2$ allows us to consider the operad  $(\Chains(\FM_2),d,\Delta)\in \mxdops$.

\begin{theorem}\label{prop:formality morphism}  The operads $\mathcal{D}_2$ and $\FM_2$ are formal as operads in mixed complexes.  Explicitly this means there exist zig-zags of quasi-isomorphisms in $\mxdops$ connecting $$(S_\ast(\mathcal D_2;\mathbb{R}),d,\Delta)\sim (S_\ast(\FM_2;\mathbb{R}),d,\Delta) \sim (\mathsf{Ger},0,R).$$

\end{theorem} 

 \begin{proof}

	Recall that the usual proof of the homotopy equivalence of $\mathcal D_2$ and $\FM_2$ \cite[Chapter 4]{FresseBook} makes use of the Boardman-Vogt W-construction to construct the following zig-zag of homotopy equivalences 
	$$\mathcal D_2 \stackrel{\sim}\leftarrow W(\FM_2) \stackrel{\sim}\to \FM_2.$$
	One readily notices that for a fixed arity both maps preserve the natural $S^1$ actions on the three topological spaces, from which it follows that $\mathcal D_2$ and $\FM_2$ are homotopy equivalent as $S^1$ operads. From the functoriality of the semi-direct product of a topological group with a topological operad it also follows \cite{GS10} that the framed versions $\mathcal D_2^{fr}$ and $\FM_2^{fr}$ are homotopy equivalent topological operads.
	
	At the algebraic level we obtain that $(S_\ast(\mathcal{D}_2),d,\Delta)\sim (S_\ast(\FM_2),d,\Delta)$, since both $\Delta$ operators are given by the composition with the unary framed element.
	
	Recall from \cite{HLTV}, that the equivalence between the functor of singular chains and the one of semi-algebraic chains is given by a zig-zag of natural quasi-isomorphisms
	$$\Chains(-) \stackrel{\sim}\leftarrow S^{PA}_\ast (-)\stackrel{\sim}\to S_\ast(-),$$
where $S^{PA}_p (X)=\{\sigma \colon \Delta^p \to X \ |\ \sigma \text{ is a semi-algebraic map}\}$. Both maps are easily seen to be compatible with the mixed complex structure, and therefore  	$(S_\ast(\mathcal{D}_2),d,\Delta)\sim (\Chains(\FM_2),d,\Delta)$. 

Kontsevich's quasi-isomorphism of operads $\Chains(\FM_2) \to \Graphs$ is compatible with the mixed complex structure, as shown in \cite[Lemma 3.1]{GS10}. It remains to see that the map $\Ger \to \Graphs$ is also compatible with the mixed complex structure. It suffices to check this statement on generators, where it is clear since $\Delta$ sends the graph with no edges in $\Graphs(2)$ to the graph contaning only two external vertices and an edge connecting them.	
\end{proof}

As a corollary to this theorem we can relate the rotational operads discussed above in Examples $\ref{Gerex}$ and $\ref{minex}$.  The proof follows immediately from the Theorem and Proposition $\ref{eqprop}$.

\begin{corollary}\label{eqcor}  There is an equivalence of rotational operads $(\mathsf{M},d,R)\sim (\mathsf{Ger}, 0, R)$.
\end{corollary}

Combining this corollary with our work in Section $\ref{secmixedcpx}$ yields the following result:

\begin{theorem}  The operad $\mathsf{M}_\circlearrowright$ of \cite{Ward2} is weakly equivalent to the gravity operad. 
\end{theorem}
\begin{proof}  By Corollary $\ref{eqcor}$ we know $(\mathsf{M},d,R)\sim (\mathsf{Ger},0,R)$ in $\rotops$.  Applying Corollary $\ref{WEtogravcor}$ to this, we find  $\tcc(\mathsf{M},R)\sim \mathsf{Grav}$ as truncated dg operads, where the commutative products generate the gravity operations.
	
	We then define a map or truncated operads $\tcc(\mathsf{M},R)\stackrel{\sim}\to \mathsf{M}_\circlearrowright\subset \mathsf{M}$ by $R$ (with $v\mapsto 0$).  This is a morphism of operads with the same homologies, and on homology it takes generators to generators so it's a quasi-isomorphism.\end{proof}

\begin{remark}  This proof works over $\mathbb{R}$.  For other fields of characteristic $0$ we can prove this result by appealing to the formality of the gravity operad along with Proposition $\ref{eqprop}$.  The analog of Theorem $\ref{prop:formality morphism}$ in the non real case is an open problem.	
\end{remark}

We recall from \cite[Theorem C]{Ward2} that the dg operad $\mathsf{M}_\circlearrowright$ acts on the cyclic deformation complex of any cyclic operad.  This action extends the Lie algebra structure discussed above (Remark $\ref{dgcycrmk}$), is compatible with the action of $\mathsf{M}$ on the (non-cyclic) deformation complex, and recovers the expected gravity structure on the homology of this complex.  We choose a weak equivalence $\mathsf{Grav}_\infty\stackrel{\sim}\to\mathsf{M}_\circlearrowright$, whose existence is guaranteed by the Theorem and then define:

\begin{definition}\label{seq1}  If $\op{O}$ is a cyclic operad with associated MC element $\mu$, we define a $\mathsf{Grav}_\infty$ structure on cyclic deformation complex $(\op{O}^\sigma, d_\mu)$ via
$\mathsf{Grav}_\infty\stackrel{\sim}\to\mathsf{M}_\circlearrowright \to End_{\op{O}^\sigma}$.  In particular, in the case $\op{O}=\OpDpoly$ this defines a $\mathsf{Grav}_\infty$ structure on $D^{\sigma}_{poly}$.
\end{definition}

\section{Cyclic Swiss Cheese type operads}\label{sec3}

In this Section we introduce the 2-colored operads that we will work with throughout the paper. They all have a compatible cyclic structure encoded by the following definition.

\begin{definition}
	Let $\cP$ be a 2-colored operad that is non-symmetric in color 2. We say that $\cP$ is of Swiss Cheese type if $\cP^1(m,n) = 0$ if $n>0$.
	
	A Swiss Cheese type operad $\cP$ endowed with a right action of the cyclic group $\mathbb Z_{n+1}$ on each $\cP^2(m,n)$ is said to be of Cyclic Swiss Cheese type (abbreviated CSC) if:
	\begin{itemize}
	\item The cyclic action is $\cP^1$ equivariant,
	
	\item The cyclic action and the color 2 compositions satisfy the same compatibility as in a cyclic operad.  
	\end{itemize}
	
	In particular, this last axiom implies that if $\cP$ is of CSC type then the partial compositions and the cyclic action combine to endow $\prod_m\cP^2(m,n)$ with the structure of a cyclic operad.
	
	A morphism of CSC type operads is a map of colored operads that is moreover equivariant with respect to the cyclic action.

\end{definition}

\subsection{Configuration spaces of points} The (original) Swiss Cheese operad is a colored operad introduced by Voronov \cite{Voronov} whose operations in color $1$ are given by rectilinear embeddings of discs in a big disc, while operations in color $2$ consist of rectilinear embeddings of discs and semi-dics in a big semi-disc. In \cite{Voronov}, Voronov considers also a homotopy equivalent operad $(\FM_2,\H )$ made out of configuration spaces of points on the plane or upper half-plane. This second construction has some advantages over the first one, one of them being that there is a natural Cyclic Swiss Cheese structure on $(\FM_2,\H )$ as we describe in this subsection.

Analogously to section \ref{sec: formality of FM2}, one can consider the configuration space of $m$ points on the upper half-plane and $n$ points at the boundary, modulo scaling and horizontal translations 
$$Conf_{m,n}(\mathbb C) = \{(x_1,\dots,x_m; y_1,\dots,y_n) \in \mathbb C^{n+m} | \Im(x_i)> 0, \Im(y_i)=0, \text{no points overlap} \}/\mathbb R_+ \ltimes \mathbb R.$$
There is an embedding $Conf_{m,n}(\mathbb C) \hookrightarrow Conf_{2m+n}(\mathbb C)$ by mirroring the bulk points along the real axis. Compactifying as in section \ref{sec: formality of FM2}, we obtain the space $\H_{m,n}$.

The spaces $\FM_2$ and $\H_{\bullet,\bullet}$ assemble into a Swiss Cheese type operad, with the two color compositions into $\H_{\bullet,\bullet}$ being still done by insertion into boundary strata.  There is in fact a cyclic action extending the Swiss Cheese structure to a Cyclic Swiss Cheese type operad structure as follows:

The open upper half plane is isomorphic to the Poincar\'e disc via a conformal map. This isomorphism sends the boundary of the plane to the boundary of the disc except one point that we label by $\infty$. We define the cyclic action of $\mathbb Z_{n+1}$ in $\H_{m,n}$ by cyclic permutation of the point labeled by infinity with the other points at the boundary.

\begin{figure}[h]
    \begin{tikzpicture}
	
    	\node (pic1) at (0,0)  
    	{\begin{tikzpicture}
    		\draw (0,0) arc (90:45:1cm) node [label= $\overline 1$] (1)[int]{} 
    		arc (45:10:1cm) node [label=right:$\overline 2$] (2)[int]{} 
    		arc (10:-45:1cm) node [label=right:$\overline 3$] (3)[int]{} 
    		arc (-45:-100:1cm) node [label=below:$\dots$]{} 
    		arc (-100:-180:1cm) node [label=left:$\overline{n-1}$] [int]{} 
    		arc (-180:-230:1cm) node [label=left:$\overline n$][int]{} 
    		arc (-230:-270:1cm) node [label=$\infty$][int]{} ;
    		\end{tikzpicture}};
    	
	    \node (text) at (2.5,0) {$\cdot \sigma\quad =$};

    	\node (pic2) at (5,0) 
    	{\begin{tikzpicture}
    		\draw (0,0) arc (90:45:1cm) node [label= $\infty$] [int]{} 
    		arc (45:10:1cm) node [label=right:$\overline 1$] [int]{} 
    		arc (10:-45:1cm) node [label=right:$\overline 2$] (3)[int]{} 
    		arc (-45:-100:1cm) node [label=below:$\dots$]{} 
    		arc (-100:-180:1cm) node [label=left:$\overline{n-2}$] [int]{} 
    		arc (-180:-230:1cm) node [label=left:$\overline {n-1}$][int]{} 
    		arc (-230:-270:1cm) node [label=above:$\overline n$][int]{};
    		\end{tikzpicture}};
    \end{tikzpicture}
\end{figure}

\subsection{Graphs}\label{sec: Gra}

For $m,n\geq 0$, let $\vKGra(m,n)$ be the free differential graded commutative algebra generated by ``edges'' $\Gamma^{i,j}$, $1\leq i,j\leq m$; ``edges'' $\Gamma^i_{\overline{j}}$, $1\leq i\leq m$; $1\leq \overline{j}\leq n$ in degree $-1$ and symbols $v_i$, $1\leq i\leq m$ of degree $-2$. 

The differential sends $v_i$ to $\Gamma^{i,i}$ and vanishes on every other generator.  The reason for the notation is that this cdga can be considered a variation of Kontsevich's graphs, used in \cite{Kontsevich}.

We interpret $\vKGra(m,n)$ as the space spanned by directed graphs with $m$ vertices of \textit{type I} labeled with the numbers $\{1,\dots, m\}$ that can be additionally decorated with a power of $v$, $n$ vertices labeled with the numbers $\{ \overline 1,\dots, \overline n\}$ of \textit{type II} and edges that can not start on a vertex of type II.

Let us consider a different free cdga, $\Gra(n)$ (c.f \cite{homotopybracesformality}), to be generated by symbols $\Gamma^{i,j}$, $1\leq i\ne j\leq n$, that is to say,  $\Gra(n)$ is the subspace of graphs of $\vKGra(n,0)$ containing no tadpoles or positive powers of $v$.  

We define a symmetric operad structure on $\Gra$ by setting the symmetric action to permute the labels on vertices and the operadic composition $\Gamma \circ_i \Gamma'$ to be the insertion of $\Gamma'$ in the $i$-th vertex of $\Gamma$ and taking a signed sum over all possible ways of connecting the edges incident to $i$ to $\Gamma$.

\begin{remark}[Sign rules]
	To obtain the appropriate signs one has to consider the full data of graphs with an ordering on the set of edges. In this situation the orientation of the edges of $\Gamma$ is preserved and one uses the symmetry relations on $\Gamma$ in such a way that the labels of the edges of the subgraph $\Gamma$ come before the labels of the edges of the subgraphs $\Gamma'$. The operad axioms are a straightforward verification. 
\end{remark}

We can form a Swiss Cheese type operad by setting $\Gra$ to be the operations in color $1$ and $\vKGra$ to be the operations in color $2$, considering the symmetric action permuting the labels of type I vertices and ignoring the symmetric action of type II vertices. The partial compositions are given as in $\Gra$, i.e., by insertion on the corresponding vertex, connecting in all possible ways and distributing corresponding the powers of $v$ also in all possible ways. 

Following Kontsevich's conventions, since type II vertices in $\vKGra$ will be seen as boundary vertices, we draw them with a line passing by the type II vertices.
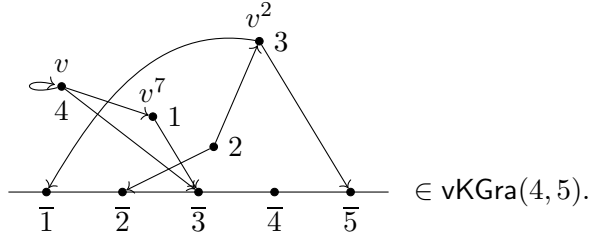
\begin{center}
	
	\begin{tikzpicture}[scale=1]
	
	\node[label=right:$1$, label=$v^7$] (1) at (1.4,1) [int] {};
	\node[label=right:$2$] (2) at (2.2,0.6) [int] {};
	\node[label=right:$3$, label=$v^2$] (3) at (2.8,2) [int] {};
	\node[label=below:$4$, label=$v$] (4) at (0.2,1.4) [int] {};
	\node[label=below:$\overline 1$] (b1) at (0,0) [int] {};
	\node[label=below:$\overline 2$] (b2) at (1,0) [int] {};
	\node[label=below:$\overline 3$] (b3) at (2,0) [int] {};
	\node[label=below:$\overline 4$]  at (3,0) [int] {};
	\node[label=below:$\overline 5$] (b4) at (4,0) [int] {};
	
	\draw [->] (1)--(b3);
	\draw [->] (2)--(b2);
	\draw [->] (2)--(3);
	\draw [->] (4)--(b3);
	\draw [->] (4) --  (1);
	\draw [->] (3)  to [out=170,in=60] (b1);
	\draw [->] (3)--(b4);
	
	\path[->] (4) edge  [loop left] ();
	
	\draw  (-0.5,0)--(4.5,0);
	
	\node at (6,0) {$\in \vKGra(4,5)$.};
	\end{tikzpicture}
	
\end{center}

We define a cyclic $\mathbb Z_{n+1} = \langle \sigma | \sigma^{n+1}=e\rangle$ action on $\vKGra(m,n)$ on generators as follows: For all $1\leq i \leq m$ and $2\leq j \leq n$, we have  $\sigma(\Gamma^i_{\bar{j}})= \Gamma^i_{\overline{j-1}}$ and 
$\sigma(\Gamma^i_{\bar{1}}) = -\sum_{k=1}^n \Gamma^i_{\bar{k}} - \sum_{k=1}^m \Gamma^{i,k}$. The action is trivial on other generators, namely for $1\leq i,j\leq m$,  $\sigma(\Gamma^{i,j})= \Gamma^{i,j}$ and $\sigma(v_i) = v_i$.

The cyclic action is extended to the whole $\vKGra(m,n)$ by requiring it to be compatible with the product in the sense that $\sigma(ab)=\sigma(a)\sigma(b), \forall a,b \in \vKGra(m,n)$.  Since $\sigma^2(\Gamma^i_{\bar{1}}) = \Gamma^i_{\bar{m}}$, we have that $\sigma^{n+1}$ acts as the identity in every one-edge graph, and therefore the action of $\mathbb Z_{n+1}$ is well defined.

\subsection{Representation of a morphism.}

Let $\cP$ be a cyclic operad and $V$ a chain complex. Notice that there is an obvious operad of Cyclic Swiss Cheese type
\begin{equation*}
\left( \End_V, \Hom\left(V^{\otimes \bullet}, \cP\right)\right)
\end{equation*}
given by insertion of functions in to tensor powers of $V$ and the cyclic operadic compositions in $\op{P}$.  If $V$ is an algebra over the operad $\op{O}$, then this induces the structure of an operad of CSC type on $\left( \op{O}, \Hom\left(V^{\otimes \bullet}, \cP\right)\right)$ in which a cross color composition is given by pushing forward along the given morphism $\op{O}\to End_V$ and then composing as in the first example.  In particular there is an induced map:
\begin{equation*}
\left( \op{O}, \Hom\left(V^{\otimes \bullet}, \cP\right)\right)\to \left( \End_V, \Hom\left(V^{\otimes \bullet}, \cP\right)\right)
\end{equation*}

Now suppose $(A,d,\Delta)$ is an algebra over $\op{O}\in\mxdops$.  Combining this example with the sequence of dg operads $\tcc(\op{O})\to \nhc(\op{O}) \to \nhc(\End_A)\to End_{\nhc(A)}$ constructed in Section 1 (via Corollaries $\ref{wecor}$ and $\ref{algcor2}$) we have a morphism of CSC type operads:
\begin{equation*}
\left( \tcc(\op{O}), \Hom\left(\nhc(A)^{\otimes \bullet}, \cP\right)\right)\to \left( \End_{\nhc(A)}, \Hom\left(\nhc(A)^{\otimes \bullet}, \cP\right)\right)
\end{equation*}

\begin{example}\label{repex}  We will make subsequent use of the following example of such an operad of CSC type.  Let $A$ be the mixed complex $(\Tpoly,0,u\Div)$, $\op{O}=End_{T_{poly}}\in \mxdops$, and $\cP=\OpDpoly$.  Then we may consider the consequent morphism of CSC type operads:
\begin{equation*}
\left( \tcc(\End_{\Tpoly}), \Hom(\nhc(\Tpoly)^{\otimes \bullet}, \OpDpoly)\right)\to \left( \End_{\nhc(\Tpoly)}, \Hom(\nhc(\Tpoly)^{\otimes \bullet}, \OpDpoly)\right)
\end{equation*}

\end{example}

\subsection{The functor $\tcc$ on operads of Cyclic Swiss Cheese type.}  

\begin{proposition}\label{prop:LCC of CSC operad}
	If $\op{P}=(\op{P}^1,\op{P}^2)$ is an operad of CSC type and if the operad $\op{P}^1$ is a rotational operad, then $\tcc(\cP)=(\tcc(\cP^1),\cP^2)$ is still an operad of CSC type, with compositions given by
	\begin{equation*}
	p_2 \  \tilde{\circ}_l \  p_1v^k=\begin{cases} p_2 \circ_l \rho(p_1) &\text{ if }k=0  \\ 0 & \text{ if } k>0, \end{cases}
	\end{equation*}
	for $p_i\in \op P^i$.\end{proposition}

\begin{proof}
	Let $p_i, p'_i\in \op P^i$. We start by showing the associativity of the composition, which is clear if we take three elements of $\op P^2$ or three elements of $\op P^1$. Otherwise, if a positive power of $v$ appears in an element of $\tcc(\cP^1)$, both double compositions will be zero and associativity holds trivially. If $p_2 \ \tilde \circ_l \  (p_1 v^0\  \tilde{\circ}_j\  p_1'v^0)$ compose sequentially (as opposed to parallely), we have $p_2 \ \tilde \circ_l \  (p_1 v^0\  \tilde{\circ}_j\  p_1'v^0)=p_2  \circ_l  \rho(p_1  \circ_j  \rho(p_1'))= p_2  \circ_l  (\rho(p_1)  \circ_j  \rho(p_1')) = (p_2  \circ_l  \rho(p_1))  \circ_{l+j-1}  \rho(p_1'))$  $=(p_2  \tilde{\circ_l}  p_1) \ \tilde\circ_{l+j-1}\  p_1'$. The other associativity verifications are straightforward.

	For the compatibility with the differential, consider that $d(p_2 \ \tilde\circ_i \ p_1v^k) = 0$ if $k>0$. In that case,
	$dp_2 \ \tilde\circ_l \ p_1v^k \pm p_2 \ \tilde\circ_l \ dp_1v^k \pm p_2 \ \tilde\circ_l \ \rho(p_1)v^{k-1}=0$, owing to the compatibility of $d$ with $\circ_i$ and the fact that $\rho^2=0$.  If $k=0$, then $dp_2 \ \tilde\circ_l \ p_1 \pm p_2 \ \tilde\circ_l \ (d+u\rho)p_1 = dp_2 \circ_l \rho(p_1) \pm p_2 \ \tilde\circ_l \ dp_1 = d(p_2 \circ_l \rho(p_1)) \mp p_2 \circ_l d\rho(p_1) \pm p_2 \circ_l \rho(dp_1) = d(p_2 \ \tilde\circ_l \ p_1)$.
	
	The cyclic action on $\op P^2$ is still $\tcc (\op P^1)$ equivariant since $$p_2^\sigma \ \tilde  \circ_l \ p_1 = p_2^\sigma \circ_l \rho(p_1) = (p_2 \circ_l \rho(p_1))^\sigma = (p_2\ \tilde  \circ_i \ p_1)^\sigma.$$
	
\end{proof}

\begin{remark}\label{rem:swiss cheese gives LCC}
	Notice that this construction defines an endofunctor on the category whose objects are operads of CSC type that are operads on mixed complexes in color $1$ and whose morphisms are maps of colored operads that are equivariant with respect to the cyclic action and commute with the rotational structure in color $1$.
\end{remark}

Recall from \cite{Ward2} the operad $\mathsf B_\circlearrowright$ constructed as the image of the rotational operator $R$ on the operad of rooted planar trees. It is the untwisted version of the operad $\mathsf M_\circlearrowright$.

\begin{proposition}\label{prop:CSC to bimod}
	Let $\op P = (\op{P}^1,\op{P}^2)$ be an operad of CSC type. The totalized space of cyclic invariants $\prod_n \Sigma^n \op P^2 (\bullet, n)^{\mathbb Z_{n+1}}$ is a $\mathsf B_\circlearrowright - \op{P}^1$ bimodule.
\end{proposition}

\begin{proof}
	The left module structure follows from \cite[Corollary 2.11]{Ward2}. The colored operad structure defines a right $\op P^1$-module structure on $\prod_n \Sigma^n\op P^2 (\bullet, n)$ and the fact that this right module structure restricts to the space of invariants is a consequence of the equivariance of $\op P^1$ with respect to the cyclic structure.
	
	The compatibility of the left and right actions follows from the associativity for parallel composition on operads, as the left action only involves insertions of color $2$ and the right action only involves insertions of color $1$. 
\end{proof}

\begin{proposition}\label{prop:CSC to bimod is functorial}
	This construction is functorial.
\end{proposition}

\begin{proof}
	The equivariance of the morphism with respect to the cyclic action guarantees that cyclic invariants are mapped to cyclic invariants.
	Since a  morphism of CSC type operads is in particular a morphism of colored operads, the induced map on the total space is a morphism of right bimodules.
	As for the right $\mathsf{B}_\circlearrowright$ action, the compatibility follows from the compatibility of the compositions with the cyclic structure, given in the axioms of a cyclic operad.
\end{proof}

\section{Bimodule maps} \label{sec4}
A homotopy $\Grav_\infty$ morphism from a $\Grav_\infty$ algebra $A$ to a $\Grav_\infty$ algebra $B$ can be expressed as a representation of the canonical $\Grav_\infty$ operadic bimodule on the colored vector space $A \oplus B$.

The strategy to find such a representation for our case $A=\Tpoly[u], B= \Dpoly^\sigma$ is to construct a certain bimodule $\mathsf{M}_\circlearrowright \circlearrowright  \ker \Delta_{\mathbb H}  \circlearrowleft   \tcc(\Chains(\mathsf{FM}_2))$ which is homotopy equivalent to the $\Grav_\infty$ canonical bimodule and construct a map of bimodules into $\End_{\Dpoly^\sigma} \circlearrowright \End^{\Tpoly[u]}_{\Dpoly^\sigma} \circlearrowleft \End_{\Tpoly[u]}$, where $\End^{\Tpoly[u]}_{\Dpoly^\sigma}(n) = \Hom({\Tpoly[u]}^{\otimes n},{\Dpoly^\sigma})  $.

More concretely, the goal of this Section is to define several operadic bimodules and construct the following series of bimodule maps, whose composition determines our desired $\Grav_\infty$ formality map.

\begin{equation}\label{diagram}
\xymatrix{ 
	\Grav_\infty \ar[d] & \circlearrowright & \ar[d]^{\ref{sec:top to top}} \Grav_\infty^{\text{bimod}}   & \circlearrowleft &  \Grav_\infty   \ar[d] \\ 
	\mathsf M_\circlearrowright \ar[d] & \circlearrowright & \ar[d]^{\ref{sec:top to top}} \ker \Delta_{\mathbb H}  & \circlearrowleft &  \tcc(\Chains(\FM_2))  \ar[d] \\ 
	\mathsf M_\circlearrowright \ar[d] & \circlearrowright & \ar[d]^{\ref{sec:top to top}} \Chains(\mathbb H_{\bullet,0})  & \circlearrowleft &  \tcc(\Chains(\FM_2))  \ar[d] \\ 
	\mathsf M_\circlearrowright \ar[d] & \circlearrowright & \ar[d]^{\ref{sec:top to graphs}} \left(\prod_n \Sigma^n\Chains(\mathbb H_{\bullet,n})^{\mathbb{Z}_{n+1}}\right)^\mu  & \circlearrowleft &  \tcc(\Chains(\FM_2))  \ar[d]\\ 
	\mathsf M_\circlearrowright \ar[d] & \circlearrowright & \ar[d]^{\ref{sec:repon col_op}} \left(\prod_n \Sigma^n \vKGra(\bullet, n)^{\mathbb{Z}_{n+1}}\right)^\mu  & \circlearrowleft &  \tcc(\Gra)  \ar[d] \\ 
	\mathsf M_\circlearrowright \ar[d] & \circlearrowright & \ar[d]^{\ref{sec:repon col_op}} \End^{\Tpoly[u]}_{\Dpoly^\sigma}  & \circlearrowleft &  \tcc(\End_{\Tpoly})  \ar[d] \\ 
	\End_{\Dpoly^\sigma} & \circlearrowright & \End^{\Tpoly[u]}_{\Dpoly^\sigma}& \circlearrowleft &  \End_{\Tpoly[u]} }
\end{equation}

Most of these maps follow from the application of Propositions \ref{prop:CSC to bimod} and \ref{prop:CSC to bimod is functorial}, sometimes after using Proposition \ref{prop:LCC of CSC operad}.  The labels on the arrows represent the section in which the respective map is constructed.

The top-most map in the diagram is due to the theory of quasi-torsors that was developed in \cite{torsors} and which we now briefly recall.

\begin{definition}
Let $\mathcal P$ and $\mathcal Q$ be two differential graded operads and let $\cM$ be a $\mathcal{P}-\mathcal Q$ operadic differential graded bimodule, i.e., there are compatible actions $$\mathcal P \circlearrowright \cM \circlearrowleft \mathcal Q.$$
We say that $\cM$ is a $\mathcal P$-$\mathcal Q$ quasi-torsor if there is an element $\mathbf 1 \in M^0(1)$ such that the canonical maps
\begin{equation}\label{modmaps}
\begin{aligned}
 l\colon \cP &\to \cM  \quad\quad\quad\quad\quad & r\colon\cQ & \to \cM \\
 p   &\mapsto p\circ (\mathbf 1,\dots, \mathbf 1)  & q &\mapsto \mathbf 1\circ q \\
\end{aligned}
\end{equation}
are quasi-isomorphisms.
\end{definition}

The main Theorem of \cite{torsors} states that if the $\cP-\cQ$-bimodule $\cM$ is an operadic quasi-torsor, then there is a zig-zag of quasi-isomorphisms connecting $\mathcal P \circlearrowright \mathcal M \circlearrowleft \mathcal Q$ to the canonical bimodule $\mathcal P \circlearrowright \cP \circlearrowleft \mathcal P$. Under good conditions, one can then homotopy lift the zig-zag to a cofibrant resolution $\mathcal P_\infty \circlearrowright \cP_\infty^{\text{bimod}} \circlearrowleft \mathcal P_\infty \to \mathcal P \circlearrowright \cM \circlearrowleft \mathcal Q$.

From this discussion it follows that to obtain $\Grav_\infty$ morphism $A\to B$ it suffices to construct a representation of a $\Grav$ quasi-torsor. We will show that the second row of diagram \eqref{diagram} is such a quasi-torsor, from which Theorem B (for  $M=\mathbb R^d$) follows. 

\subsection{From topology to graphs}

Recall from Section \ref{sec: formality of FM2} the map of cooperads $\omega_\bullet\colon \Gra^* \to \Omega(\FM_2)$. We wish to define a similar map $\omega_\bullet\colon \vKGra^*\to \Omega(\mathbb H_{\bullet,\bullet})$ denoted by the same symbol by abuse of notation.

Let us consider the (multivalued) angle function $\theta$ on $\mathbb H_{m,n}$ such that
\begin{equation*}
\theta(z,w,x) = \frac{1}{2 \pi} \arg \left( \frac{(w-z)(1-\bar{z}x)}{(1-\bar{z}w)(x-z)} \right) \end{equation*}
giving the angle between the geodesics $[w,z]$ and $[z,x]$. Since all values differ by an integer, the differential $d\theta$ is a well-defined $1$-form.

The map $\omega_\bullet\colon \vKGra^*\to \Omega(\mathbb H_{\bullet,\bullet})$ is defined to be a map of commutative algebras as follows:

\begin{itemize}
	\item The one-edge graph $\Gamma^{i,j}\in \vKGra^*$ \footnote{We identify the basis of $\vKGra$ with its dual basis in $\vKGra^*$, except the dual of the elements $v$ that we denote by $u$.} for $i\ne j$ is sent to $d\phi^{i,j}\coloneqq d \theta(z_i,z_j,z_\infty)$. $\phi^{i,j}$ can be pictured as the hyperbolic geodesic passing through  $i$ and $j$ and the vertical line passing by $i$ or alternatively, on the hyperbolic disc, this angle can be pictured as the angle between the lines $[\infty, i]$ and $[i,j]$.
	
	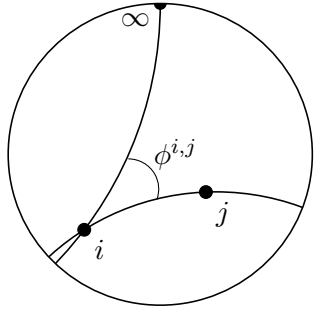
\begin{figure}[h]
		\begin{tikzpicture}[scale=2]
		\tkzDefPoint(0,0){O}
		\tkzDefPoint(1,0){A}
		\tkzDrawCircle(O,A)
		\tkzDefPoint(0.3,-0.25){j}
		\tkzDefPoint(-0.5,-0.5){i}
		\tkzDefPoint(0,1){z3}
		\tkzClipCircle(O,A)
		\tkzDrawCircle[orthogonal through=i and j](O,A)
		\tkzDrawCircle[orthogonal through=i and z3](O,A)
		\tkzDrawPoints[color=black,fill=black,size=12](i,j)
		\tkzLabelPoints(i,j)
		
		\fill (0,1) circle[radius=1.2pt] node[below left] {$\infty$};

		\node at (0.1,0) {$\phi^{i,j}$}; 
		\draw (-0.21,-0.03) arc (90:-20:0.2);
		\end{tikzpicture}
		\caption{The hyperbolic angle $\phi^{i,j}$}
	\end{figure}

	\item Similarly, the one-edge graphs $\Gamma^i_{\overline{j}}\in \vKGra^*$ are sent to $d\phi^i_{\overline{j}} \coloneqq d\theta(z_i,z_{\overline{j}},z_\infty)$.
	
	\item The tadpole graphs $\Gamma^{k,k}$ are sent to Willwacher's form $\eta_{z_k}$, where 
	$$\eta_{z} = \sum_{i=0}^n \theta(z,z_{\overline{i+1}}, z_{\bar{i}}) d\theta(z,z_{\bar{i}}, z_{\infty}),$$
	Notice that $\theta(z,z_{\overline{i+1}}, z_{\bar{i}})$ is a well defined smooth function since the points $z_{\bar{i}}$ and $z_{\overline{i+1}}$ are on the boundary of the disc.
	
	\item To define the image of a graph $\Gamma$ with no edges and with a vertex decorated with $u=v^*=v^{-1}$, in order for $f_2$ to commute with the differential, we have no choice but to define 
	$f_2(v_k) = d\eta_{z_k}$.
	
	\end{itemize}
	
	\begin{remark}\label{rem: not CSC map}
		Recall from Theorem \ref{prop:formality morphism} that the map $\Chains(\mathsf{FM}_2)\to \Gra$ is compatible with the mixed complex structure and therefore induces a map $\tcc(\Chains(\mathsf{FM}_2))\to \tcc(\Gra)$. One would like to have a map of CSC type operads $$\left(\Chains(\mathsf{FM}_2),  \Chains(\mathbb H_{\bullet,\bullet})\right) \to (\Gra, \vKGra),$$ however, due to the existence of tadpoles in $\vKGra$ this map is not compatible with the operadic composition. The next proposition states that we obtain nevertheless a map of CSC operads after taking the functor $\tcc$.
	\end{remark}

	\begin{proposition}\label{prop: need CCtheta}
		The map $$\omega_\bullet^*\colon \left(\tcc (\Chains(\mathsf{FM}_2)),  \Chains(\mathbb H_{\bullet,\bullet})\right) \to (\tcc (\Gra), \vKGra)$$
		is a morphism of operads of Cyclic Swiss Cheese type.
	\end{proposition}
	
	\begin{proof}
		Let us start by showing that the map $\vKGra^*\to \Omega(\mathbb H_{\bullet,\bullet})$ is compatible with the cyclic structure. It suffices to check this on the generators of $\vKGra^*$. For this, we start by observing that the cyclic structure on $\vKGra^*$, being dual to the one of $\vKGra$ is the following:
	
		\begin{itemize}\renewcommand{\labelitemii}{$\bullet$}
			\item $\left(\Gamma^{i,j}\right)^\sigma =\Gamma^{i,j} -\Gamma^i_{\overline{1}}$,
			\item $\left(\Gamma^i_{\bar{j}}\right)^{\sigma} = \Gamma^i_{\overline{j+1}}-\Gamma^i_{\overline{1}}$ with the convention that $\Gamma^i_{\overline{n+1}}=0$,
			\item $(v_i)^{\sigma} = v_i$.
		\end{itemize}
		
		The cyclic structure on $\Omega(\mathbb H_{m,n})$ is given by the pullback of the cyclic structure on $\mathbb H_{m,n}$. It follows that $\theta(z,z_{\overline{i}},z_{\overline{j}}) = \theta(z,z_{\overline{i+1}},z_{\overline{j+1}})$, from which it follows that $$d{\phi^{i}_{\overline{j}}}^\sigma =  d\theta(z_i,z_{\overline{j+1}},z_{\overline{1}}) = d\theta(z_i,z_{\overline{j+1}},z_{\infty})- d\theta(z_i,z_{\overline{1}},z_{\infty}) = d\phi^{i}_{\overline{j+1}} - d\phi^i_{\overline{1}}$$
		
		Similarly, it follows that $d{\phi^{i,j}}^\sigma = d\phi^{i,j} - d\phi^i_{\overline 1}$.
		
		To show that $\eta_{z_i}^\sigma = \eta_{z_i}-d\phi^i_{\overline 1}$, we make use of the fact that $\sum_{i=0}^n\theta(z,z_{\overline{i+1}},z_{\overline{i}}) =1$.
		
		\begin{align*}
		\eta_z^\sigma &= \sum_{i=0}^n\theta(z,z_{\overline{i+2}},z_{\overline{i+1}})d\theta(z,z_{\overline{i+2}},z_{\overline{1}})\\
		&=\sum_{i=0}^n\theta(z,z_{\overline{i+1}},z_{\overline{i}})d\theta(z,z_{\overline{i+1}},z_{\overline{1}})\\
		&=\sum_{i=0}^n\theta(z,z_{\overline{i+1}},z_{\overline{i}})(d\theta(z,z_{\overline{i+1}},z_{\infty})-d\theta(z,z_{\overline{1}},z_{\infty}))\\
		&=\eta_z - d\phi^z_{\overline{1}}. 
		\end{align*}
		
		Notice that the same computation, but with $d\theta$ instead of $\theta$ shows that $d\eta_z$ is invariant by the cyclic action.
		\\

		To show the compatibility with the cooperadic structure, let us start by noticing that the pullback of forms of the type $\eta_{z_k},d\phi^{i,j}$ and $d\eta_{z_k}$ under the composition map $\circ_i\colon \mathbb H_{m,n} \times \FM_2(k) \to \mathbb H_{m+k-1,n}$ is expressible with forms of the same type. For instance, the image of $d\phi^{1,2}\in \Omega(\mathbb H_{m,n})$ inside $\Omega(\mathbb H_{m-2,n}\times \FM_2(3))$ under the map $\circ_1^*$ is the form $1\otimes d\phi_{1,2} \in \Omega(\mathbb H_{m-2,n})\otimes \Omega(\FM_2(3))\subset \Omega(\mathbb H_{m-2,n}\times \FM_2(3))$, while $\circ_1^*(\eta_{z_1}) = \eta_{z_1} \otimes 1$.

		Let $X\in \Chain_l(\mathbb H_{m,n})$ and $Y\in  \Chain_r(\FM_2)$. The operadic compatibility in mixed colors amounts to showing that
		$$\sum_{\Gamma\in \vKGra_{l+r}(m+k-1,n)} \left(\int_{X\circ_i \Delta(Y)} \omega_{\Gamma}\right) \cdot \Gamma = \sum_{\tilde{\Gamma}\in \vKGra_l(m,n)} \left(\int_{X} \omega_{\tilde{\Gamma}}\right) \cdot \tilde{\Gamma} \circ_i \sum_{\Gamma'\in \Gra_{r}(k)} \left(\int_{Y} \omega_{\Gamma'}\right) \cdot \Delta(\Gamma'). \footnote{The sums are meant to be taken over the basis of graphs.}$$

		We need therefore to show that for every graph $\Gamma$ in $\vKGra_{l+r}(m+k-1,n)$, we have 
		
		\begin{equation}\label{eq:chains to graphs}
		\int_{X\circ_i \Delta(Y)} \omega_{\Gamma} = \sum_{\substack{\Gamma'\in \Gra_{r}(k)\\ \tilde{\Gamma}\in \vKGra_l(m,n) \\ \Gamma\in \tilde{\Gamma} \circ_i \Delta(\Gamma')}}\left(\int_X \omega_{\tilde\Gamma}\right)\left(\int_Y \omega_{\Gamma'}\right).
		\end{equation}
		
		Consider those vertices in $\Gamma$ labeled with numbers $i,i+1,\dots,i+k-1$ and let $\gamma$ be the subgraph of $\Gamma$ induced by these vertices. Furthermore, let us consider $\gamma^{\notadp}$ the subgraph of $\gamma$ where we disregard tadpoles and powers of $v$.
		
		Notice that $\int_{X\circ_i \Delta(Y)} \omega_{\Gamma} = \int_{X\times \Delta(Y)} \circ_i^*\omega_{\Gamma}$ and that we can decompose this integral into $\int_{X} \omega_{rest} \cdot \int_{\Delta(Y)}  \omega_{\gamma^{\notadp}}$. Notice that if $\gamma$ has at least two tadpoles, there will be two copies of $\eta_{z_i}$ in $\omega_{rest}$ and therefore $\omega_{rest}=0$. If that is the case, then also the right hand side of \eqref{eq:chains to graphs} must vanish as $\tilde\Gamma$ would need to have two tadpoles at the vertex $i$.  Here we have used the fact that $\Gamma^\prime$ and hence $\Delta(\Gamma^\prime)$ are tadpole free (after Example $\ref{graex}$).
		
		Suppose now that $\gamma$ has one tadpole. Then, a decomposition $\Gamma \in   \tilde{\Gamma} \circ_i \Delta(\Gamma')$ allows just one choice of $\tilde \Gamma$ (that requires $\tilde \Gamma$ to have a tadpole at $i$ and a power of $v$ equal to the total amount of powers of $v$ in $\gamma$).
		
		It suffices to check that $\int_{\Delta(Y)} \omega_{\gamma^\notadp} = \sum_{\Gamma'}\int_{Y} \omega_{\Gamma'}$, where the sum is being taken over the admissible $\Gamma'$ such that we find $\Gamma$ as a summand in $\tilde \Gamma \circ_i \Delta(\Gamma')$. Since $\Delta(\Gamma')$ adds an edge in every possible way, those admissible graphs correspond precisely to all possible graphs that one obtains by removing one edge from $\gamma^{\notadp}$. Notice however that $\int_{\Delta(Y)} \omega_{\gamma^\notadp} = \int_Y \Delta^*(\omega_{\gamma^\notadp})$ and $\Delta^*$ is a derivation that sends every $d\phi_{i,j}$ to the constant function $1$. This is because the projection maps $p_{ij}\colon \FM_2(n) \to \FM_2(2)$ are $S^1$ equivariant and therefore $\Delta^*(d\phi_{ij}) = p^*(\Delta^*(d\phi_{12})) = p^*(1) = 1$. It follows that expanding $\Delta^*(\omega_{\gamma^\notadp})$ one obtains exactly the same admissible graphs $\Gamma'$.

		Suppose now that $\gamma$ has no tadpoles. In general there are two possibilities for the choice of $\tilde\Gamma$, one containing a tadpole at $i$ and other not containing a tadpole at $i$. Suppose we consider $\tilde\Gamma$ with a tadpole at $i$. Then, for every admissible choice of $\Gamma'$, when we compute $\tilde{\Gamma} \circ_i \Delta(\Gamma')$ two copies of $\Gamma$ appear with opposite signs, since, like $\Delta$, inserting at a tadpole vertex produces every possible edge. Therefore there is no contribution on the right hand side of \eqref{eq:chains to graphs} if we take $\tilde\Gamma$ containing a tadpole at $i$. If $\tilde \Gamma$ contains no tadpole at $i$, then, as before, $\omega_{rest}=\omega_{\tilde{\Gamma}}$ and $\Delta^*(\omega_{\gamma^{\notadp}}) = \displaystyle\sum_{\text{admissible } \Gamma'} \omega_{\Gamma'}$.

		The compatibility  of $\omega_\bullet$ with the composition in color $1$ is clear.
	\end{proof}


By applying Propositions \ref{prop:CSC to bimod} and \ref{prop:CSC to bimod is functorial} we obtain the following Corollary.
\begin{corollary}\label{cor:bimod map}
	There exist bimodules and a bimodule morphism
	\begin{equation*}
	\xymatrix{ 
		\mathsf{B}_\circlearrowright \ar[d]_{\text{id}} & \circlearrowright & \ar[d] \prod_n \Sigma^n \Chains(\mathbb H_{\bullet,n})^{\mathbb{Z}_{n+1}} & \circlearrowleft & \ar[d] \tcc(\Chains(\mathsf{FM}_2)) \\ 
		\mathsf{B}_\circlearrowright & \circlearrowright & \prod_n \Sigma^n \vKGra(\bullet, n)^{\mathbb{Z}_{n+1}} & \circlearrowleft &  \tcc(\Gra) }
	\end{equation*}
\end{corollary}

\subsection{Twisting left modules}\label{sec:top to graphs}
Let us consider $\mu$, the $n-2$ dimensional chain of $\mathbb H_{0,n}$ in which the points at the boundary are free, the ``fundamental chain of the boundary". This chain is invariant under the $\mathbb{Z}_{n+1}$ action and therefore defines a degree $2$ element of $\prod_n \Sigma^n\Chains(\mathbb H_{\bullet,n})^{\mathbb{Z}_{n+1}}$. This is a Maurer--Cartan element with respect to the $\mathfrak s^{-1}\Lie$ action induced by $\mathfrak s^{-1}\Lie\to \mathsf{B}_\circlearrowright$. We can therefore twist the left modules of the diagram from Corollary \ref{cor:bimod map} by $\mu$ and its image $\mu'$ to obtain the following bimodule map

\begin{equation*}
\xymatrix{ 
	\Tw\mathsf{B}_\circlearrowright \ar[d]_{\text{id}} & \circlearrowright & \ar[d] \left(\prod_n \Sigma^n\Chains(\mathbb H_{\bullet,n})^{\mathbb{Z}_{n+1}}\right)^\mu  \\ 
	\Tw\mathsf{B}_\circlearrowright & \circlearrowright & \left(\prod_n\Sigma^n \vKGra(\bullet, n)^{\mathbb{Z}_{n+1}}\right)^{\mu}}
\end{equation*}

Since the left action concerns boundary points and the right action concerns bulk points, the two actions are compatible, giving us the bimodule map

\begin{equation}\label{eq:twisted bimod map 1}
\xymatrix{ 
	\Tw\mathsf{B}_\circlearrowright \ar[d]_{\text{id}} & \circlearrowright & \ar[d] \left(\prod_n \Sigma^n \Chains(\mathbb H_{\bullet,n})^{\mathbb{Z}_{n+1}}\right)^\mu  & \circlearrowleft &  \tcc(\Chains(\mathsf{FM}_2))  \ar[d] \\ 
	\Tw\mathsf{B}_\circlearrowright & \circlearrowright & \left(\prod_n \Sigma^n \vKGra(\bullet, n)^{\mathbb{Z}_{n+1}}\right)^{\mu}& \circlearrowleft &  \tcc(\Gra) }
\end{equation}

We can also consider the restriction of the left actions to $\mathsf{M}_\circlearrowright$, giving us the fourth map in \eqref{diagram}.

\subsection{Topological maps}\label{sec:top to top}

The projection map $p\colon \mathbb H_{m,n} \to \mathbb H_{m,0}$ that forgets the points at the boundary  induces a  strongly continuous chain \cite{HLTV} $p_{m,n}^{-1} \colon \mathbb H_{m, 0} \to \Chains(\mathbb H_{m, n})$.
The image of a configuration of points in $\mathbb H_{m, 0}$ can be interpreted as the same configuration of points but with $n$ points at the real line that are freely allowed to move. If we consider the total space $\Chains(\mathbb H_{\bullet, 0}) = \bigoplus_{m\geq 1}\Chains( \mathbb H_{m, 0})$, this induces a degree preserving map 
$$p^{-1} \colon \Chains(\mathbb H_{\bullet, 0}) \to \displaystyle\prod_{n\geq 0} \Sigma^n\Chains^{\mu}(\mathbb H_{\bullet, n}).$$

Notice that this map actually lands in the cyclic invariant space $\left(\prod_{n\geq 0} \Sigma^n\Chains^{\mathbb Z_{n+1}}(\mathbb H_{\bullet, n})\right)^{\mu} $

\begin{proposition}
	The map $p^{-1}$ is a morphism of right $\tcc(\Chains(\mathsf{FM}_2))$ modules and its image is stable under the action of $\mathsf{M}_\circlearrowright$.
\end{proposition}

The proof of this result is essentially in \cite[Appendix A.2]{homotopybracesformality} where the reader can find further details.

\begin{proof}
	The morphism clearly commutes with the right action.
	Let $c\in \Chains(\mathbb H_{m, 0})$. 
	
	The boundary term $\partial p_{m,n}^{-1}(c)$ has two kind of components. When at least two points at the upper half plane get infinitely close, giving us the term $p_{m,n}^{-1}(\partial c)$, and when points at the real line get infinitely close, giving us $\pm p_{m,n}^{f\partial}(c)$, where the $f\partial$ superscript represents that we are considering the boundary at every fiber.
	
	Then, we have
	$p^{-1}(\partial c) = \prod_{n\geq 0} p_{m,n}^{-1}(\partial c) =  \prod_{n\geq 0} \partial p_{m,n}^{-1}(c)  \pm p_{m,n}^{f\partial}(c)$. The first summand corresponds to the normal differential in $\Chains(\mathbb H_{m, n})$ and the second summand is precisely the extra piece of the differential induced by the twisting.
	
	It remains to check the stability under the left $\mathsf{M}_\circlearrowright\subset \mathsf{M}$ action. In fact, the stronger statement that the image is stable by the $\mathsf{M}$ holds. To show this, it is enough to check the stability under the action of the generators $T_n$ and $T_n'$.
	
	Let $c_0,\dots, c_n$ be chains on $\Chains(\mathbb H_{m,0})$ and consider the action of generators of the form $T_n\in \mathsf{M}(n+1)$ on their images, i.e. consider $T_n(p^{-1}(c_0), p^{-1}(c_1),\dots,p^{-1}(c_n))$. The result follows from computing that 
	$$T_n(p^{-1}(c_0), p^{-1}(c_1),\dots,p^{-1}(c_n)) = p^{-1}(p(T_n(p^{-1}(c_0), p^{-1}(c_1),\dots,p^{-1}(c_n))))$$
	
	and a similar equality for $T_n'$.
\end{proof}

Since $p^{-1}$ is right inverse to the projection map, from this proposition it follows that $\Chains(\mathbb H_{\bullet,0} )$ has a natural left $\mathsf{M}_\circlearrowright$ module structure. This gives us the third map of bimodules from diagram \eqref{diagram}

\begin{equation*}
\xymatrix{ 
	\mathsf{M}_\circlearrowright \ar[d]_{\text{id}} & \circlearrowright & \ar[d]  \Chains(\mathbb H_{\bullet,0} ) & \circlearrowleft &  \tcc(\Chains(\mathsf{FM}_2))  \ar[d] \\ 
	\mathsf{M}_\circlearrowright & \circlearrowright & \left(\prod_n \Sigma^n \Chains(\mathbb H_{\bullet,n})^{\mathbb{Z}_{n+1}}\right)^{\mu'}  & \circlearrowleft &  \tcc(\Chains(\mathsf{FM}_2))  }
\end{equation*}

We want now to make the first row a quasi-torsor. The left and right operads have the correct homology $\Grav$ however, as a symmetric sequence $ H(\mathbb H_{\bullet,0} )=\Ger$. 

Notice that there is no analog of the $S^1$ action of $\FM_2$ on $\mathbb H_{\bullet,0}$.  We can nevertheless define a mixed complex structure at the chain level in the following way.  Let $i\colon \FM_2 \to \mathbb H_{\bullet,0} $ be the map resulting from collapsing a configuration into one point, or alternatively, composing a configuration in $\FM_2$ with the single element $\mathbbm{1}\in \mathbb H_{1,0} $.  This is a homotopy equivalence and admits a retract $r\colon  \mathbb H_{\bullet,0}\to \FM_2$ by forgetting the boundary line.  In particular $ri = \text{id}$.





Denoting the induced maps on chains also by $i$ and $r$, we see $\Chains(\mathbb H_{\bullet,0})$ has a mixed complex structure by defining the degree $1$ map $\Delta_{\mathbb H} \colon \Chains(\mathbb H_{\bullet,0}) \to \Chains(\mathbb H_{\bullet,0})$ to be $\Delta_{\mathbb H} = i \Delta_{\FM_2} r$. From $ri = \text{id}$ it follows $\Delta_{\mathbb H}^2=0$.

\begin{proposition}
	The subspace $\ker \Delta_{\mathbb H}  \subset \Chains(\mathbb H_{\bullet,0})$ is a  $\mathsf{M}_\circlearrowright-\tcc(\Chains(\mathsf{FM}_2))$ sub-bimodule.
\end{proposition}

\begin{proof}
	Let $h\in \Chains(\mathbb H_{\bullet,0})$ and let $c\in \Chains(\FM_2)$ so that $v^kc\in \tcc(\Chains(\FM_2))$. 
	
	
	We have $\Delta_{\mathbb H} (h\circ_i c)=i\Delta r(h\circ_i c)= i\Delta(r(h) \circ_i c) = i(\Delta r(h) \circ_i c) + ir(h)\circ_i \Delta(c) = \Delta_{\mathbb H}(h) \circ_i c + ir(h)\circ_i \Delta(c) $.
	
	Therefore,  
	$\Delta_{\mathbb H} (h\ \tilde\circ_i\ c) = \Delta_{\mathbb H} ( h \circ_i \Delta c) = \Delta_{\mathbb H}  h \circ_i \Delta c$, so if $h\in \ker \Delta$, also $h\ \tilde\circ_i\ c\in \ker \Delta$.
	For higher powers of $k$ we have  $\Delta(h\ \tilde\circ_i\ v^kc)=0$, therefore $\ker \Delta_{\mathbb H} $ is trivially stable by the right action.
	
	On the other hand, $\ker \Delta_{\mathbb H} $ is stable by the left action since it only involves cyclic operations and compositions, all of which are compatible with $\Delta_{\mathbb H} $.
	
\end{proof}

\begin{proposition}
	$$\xymatrix{ 
		\mathsf{M}_\circlearrowright & \circlearrowright &  \ker \Delta_{\mathbb H} & \circlearrowleft &  \tcc(\Chains(\mathsf{FM}_2))}$$
	is a quasi-torsor.
\end{proposition}
\begin{proof}  We may apply Corollary \ref{wecor} to the underlying $\mathbb{S}$-modules to find 
\begin{equation*}
 H(\ker \Delta_{\mathbb H})= \hnhc (\Chains(\mathsf{FM}_2)) = \Grav = H(	\mathsf{M}_\circlearrowright),   
\end{equation*}
 and we just need to check that the maps $p\colon \mathsf{M}_\circlearrowright \to \ker \Delta_{\mathbb H}$ and $q\colon\tcc(\Chains(\mathsf{FM}_2)) \to \ker \Delta_{\mathbb H}$ induce quasi-isomorphisms.
	
	The map $q$ fits in the following commutative diagram
	$$
	\begin{tikzcd}
	\tcc(\Chains(\FM_2)) \arrow{r}{q}\arrow{d}{\theta^{-1}} & \ker \Delta_{\mathbb H} \arrow[hookrightarrow]{d}\\
	\nhc(\Chains(\FM_2)) \arrow{r}{\lhc^-(i)} & \lhc^-(\Chains(\mathbb H_{\bullet,0}))
	\end{tikzcd}$$ 
	Since all other maps are quasi-isomorphisms, so is $q$.
	
	To see that $p\colon \mathsf{M}_\circlearrowright \to \ker \Delta_{\mathbb H}$ is a quasi-isomorphism, notice that since we already know the homologies are isomorphic, it suffices to show that the generators $g_n \in \mathsf M^{\pm 1}_\circlearrowright (n)$ are sent to generators of the homology of $\ker \Delta_{\mathbb H} (n)$.
	
	Notice that since $\Grav_1(n)$ is 1-dimensional, in fact in suffices to show that $p(g_n)$ is non-zero in homology. 
	
	For this, notice that if we denote by $\angle_{12}\colon \mathbb H_{n,0} \to S^1$ the map remembering only the angle between points $1$ and $2$, then, the image of $p(g_n)$ under the composition $\ker \Delta_{\mathbb H} \hookrightarrow C(\mathbb H_{n,0}) \stackrel{\angle_{12}}{\to} C(S^1)$ is (homologous to) the fundamental chain of the circle and is therefore non-zero. 
\end{proof}

\begin{remark} 
	This is another way to show the formality of $\mathsf{M}_\circlearrowright$, as it would be now quasi-isomorphic to $\tcc(\Chains(\mathsf{FM}_2))$ and therefore to $\tcc(\Graphs)$ and  $\Grav \to \tcc(\Graphs)$ is a quasi-isomorphism.
\end{remark}

\subsection{Action of Graphs on $\Tpoly[u]$ and $\Dpoly^\sigma$}\label{sec:repon col_op}

In this Section we construct the action on $\Tpoly[u]$ and $\Dpoly^\sigma$. We express this in the form of an operadic bimodule morphism

\begin{equation*}
\xymatrix{ 
	\mathsf M_\circlearrowright \ar[d] & \circlearrowright & \ar[d] \left(\prod_n \Sigma^n \vKGra(\bullet, n)^{\mathbb{Z}_{n+1}}\right)^{\mu'}  & \circlearrowleft &  \tcc(\Gra)  \ar[d] \\ 
	\End_{\Dpoly^\sigma} & \circlearrowright & \End^{\Tpoly[u]}_{\Dpoly^\sigma}& \circlearrowleft &  \End_{\Tpoly[u]} }
\end{equation*}

\begin{remark}
	If one tries to replicate the arguments of the previous section, the starting place would be to construct a map of CSC type operads
	$$ \left(\Gra,\vKGra\right) \to \left( \End_{\Tpoly}, \Hom\left(\Tpoly[u]^{\otimes \bullet}, \Dpoly \right)\right), $$
	and the argument would continue with the application of Proposition \ref{prop:LCC of CSC operad}.
	Unfortunately, on the right hand side we don't have a colored operad due to the non-compatibility of the differential with the operadic composition. 
	
	We will rectify this problem by using the operad of CSC type 
	\begin{equation*}
	\left( \tcc(\End_{\Tpoly}), \Hom\left(\Tpoly[u]^{\otimes \bullet}, \OpDpoly\right)\right) 
	\end{equation*}
	(see Example $\ref{repex}$) as an intermediary.
\end{remark}

Recall Kontsevich's action of the operad $\Gra$ on $\Tpoly$ \cite{Kontsevich} given for every graph $\Gamma \in \Gra(k)$ and vector fields $X_1,\dots,X_k \in \Tpoly(\mathbb R^d)$ by
$$\Gamma(X_1,\dots, X_k) = \left(\prod_{(i,j)\in\Gamma} \sum_{l=1}^d \frac{\partial}{\partial x_l^{(j)}}\wedge \frac{\partial}{\partial \xi_l^{(i)}} \right) (X_1\wedge\dots\wedge X_k),$$

where $x_1,\dots, x_d$ are the coordinates in $\mathbb R^d$ and $\xi_1, \dots, \xi_d$ be the corresponding basis of vector fields. Notice that this map is compatible with the mixed complex structure on both sides.

\begin{proposition}\label{prop:graphs to Tpoly}
	There is a map of Cyclic Swiss Cheese type operads
	$$ \left(\tcc(\Gra),\vKGra\right) \to \left( \tcc(\End_{\Tpoly}), \Hom\left(\Tpoly[u]^{\otimes \bullet}, \OpDpoly\right)\right). $$
\end{proposition}

\begin{proof}
	The map $\tcc(\Gra)\to \tcc(\End_{\Tpoly})$ is obtained by taking the functor $\tcc$ to Kontsevich's map above and is therefore a map of dg operads. The map $\vKGra\to \Hom\left(\Tpoly[u]^{\otimes \bullet}, \OpDpoly\right)$ is essentially\footnote{Notice that Willwacher's graphs do not involve powers of $v$.} defined as described in \cite[Section 4.2]{WC}. 
	For $X_1u^{i_1},\dots,X_mu^{i_m}\in\Tpoly[u]$  the action of $\Gamma \in \vKGra(m,n)$ on $X_1u^{i_1},\dots,X_mu^{i_m}$ is zero if there exists a vertex $l$ of type I in $\Gamma$ such that the power of $v$ at the vertex $l$ does not match $i_l$. Otherwise, for $f_1,\dots,f_n\in C_c^\infty (\mathbb R^d)$, the action is given by 
	\begin{equation}\label{eq:action of graphs}
	\Gamma(X_1u^{i_1},\dots,X_mu^{i_m})(f_1,\dots, f_n) =  \left(\prod_{(i,j)\in\Gamma} \sum_{r=1}^d \frac{\partial}{\partial x_r^{(j)}}\wedge \frac{\partial}{\partial \xi_r^{(i)}} \right) (X_1,\dots, X_m;f_1,\dots, f_n),
	\end{equation}
	where the product runs over all edges of $\Gamma$ in the order given by the numbering of edges and the superscripts $(i)$ and $(j)$ mean that the partial derivative is being taken on the $i$-th and $j$-th component of $X_1,\dots,X_m$ (or $f_j$, if $j$ corresponds to a type II vertex).
	
	We need to check compatibility with the differentials. For simplicity of notation, let us focus on the piece of the differential acting on the vertex $1$ and suppose this is decorated by $v^k$ and let us denote by $d_1$ the piece of the differential only acting on the first vertex, i.e., the piece that lowers $k$ by $1$ and adds a tadpole. Since the differential on $\OpDpoly$ is zero, we need to show that $0=d_1\Gamma(X_1u^{i_1},\dots) - \Gamma(d(X_1u^{i_1}),\dots)$. Both summands are zero if $k\ne i_1-1$ and if $k=i_1-1$ they cancel since the action of a tadpole on a multivector field produces its divergence.
	
	The compatibility of the map with the mixed color composition is clear, as the map $\Gra \to \End(\Tpoly)$ is given by essentially the same formula \eqref{eq:action of graphs}.
	
	To check the compatibility with the cyclic action in color $2$ we notice that the cyclic action on $\Hom\left(\Tpoly[u]^{\otimes \bullet}, \OpDpoly\right)$ is given by the cyclic action on $\OpDpoly$ and integration by parts produces exactly the kind of graphs given by the cyclic action on $\vKGra$. An explicit computation can be found in \cite[Lemma 20]{Campos}.
\end{proof}
Combining this result with Example $\ref{repex}$ we find:

\begin{corollary}
There is a map of Cyclic Swiss Cheese type operads
\begin{equation*}
    \left(\tcc(\Gra),\vKGra\right)\longrightarrow \left(\End_{\Tpoly[u]},\Hom\left(\Tpoly[u]^{\otimes \bullet}, \OpDpoly\right)\right)
\end{equation*}
\end{corollary}

Applying Propositions \ref{prop:CSC to bimod} and \ref{prop:CSC to bimod is functorial} to 
this result, we get the following bimodule maps:

\begin{equation*}
\xymatrix{ 
	\mathsf B_\circlearrowright \ar[d] & \circlearrowright & \ar[d] \left(\prod_n \Sigma^n \vKGra(\bullet, n)^{\mathbb{Z}_{n+1}}\right)  & \circlearrowleft &  \tcc(\Gra)  \ar[d] \\ 
	\mathsf B_\circlearrowright \ar[d] & \circlearrowright & \ar[d] \End^{\Tpoly[u]}_{\TDpoly^\sigma}  & \circlearrowleft &  \tcc(\End_{\Tpoly})  \ar[d] \\ 
	\mathsf B_\circlearrowright & \circlearrowright & \End^{\Tpoly[u]}_{\TDpoly^\sigma}& \circlearrowleft &  \End_{\Tpoly[u]} }
\end{equation*}

By twisting we get the last 3 rows from diagram \eqref{diagram}. Notice that we can replace the last row by $\End_{\Dpoly^\sigma}$ as the action on the bimodule is, by definition, the action on $\Dpoly^\sigma$.

\subsection{An extension of Willwacher's morphism}

In this section we remark that the restriction of the $\Grav_\infty$ morphism to $\sLie_\infty$ is Willwacher's morphism \cite{WC}. This follows essentially from the rigidty of bimodule maps from $\sLie_\infty^{\text{bimod}}$ to $\Chains(\mathbb H_{\bullet,0})$. Concretely, suppose we take two maps $f,g\colon \sLie_\infty^{\text{bimod}} \to \Chains (\mathbb H_{\bullet,0})$ that agree in arity $1$ (notice that $\mathbb H_{1,0}) = \{pt\}$). An inductive argument shows that then $f$ and $g$ must be the same map. 

Let us consider the family $(\mu_n)_{n\geq 1}$ of generators of $\sLie_\infty^{\text{bimod}} $. The element  $\mu_n \in \sLie_\infty^{\text{bimod}} (n)$ has degree $2-2n$. 
Assume by induction that $f(\mu_k) = g(\mu_k)$ for all $k<n$. Then $d(f(\mu_n)-g(\mu_0)) = f(d\mu_n) - g(d\mu_n)= 0$, since the differential of $\mu_n$ only involves elements $\mu_k$ for $k<n$. Therefore $f(\mu_n)-g(\mu_n)$ represent a homology class in $H_{2-2n}(\mathbb H_{n,0}) = \sLie_{2-2n}^{\text{bimod}}(n) = 0$. 

It follows that there exists some chain $c\in \Chain_{1-2n}(\mathbb H_{n,0})$ such that $dc = f(\mu_n)-g(\mu_n)$, but since $\dim(\mathbb H_{n,0}) = 2n-2$ there can be no such (non-zero) chain $c$ from which our conlusion follows.

\section{Globalization}\label{sec5}
Let $M$ be a $d$-dimensional oriented manifold.
In this section we show that the $\Grav_\infty$ quasi-isomorphism $\Tpoly[u](\mathbb R^d) \to \Dpoly^\sigma(\mathbb R^d)$ constructed in the previous sections can be  globalized to a quasi-isomorphism $\Tpoly[u](M) \to \Dpoly^\sigma(M)$. All work is essentially already done as the globalized version follows from formal geometry techniques as in the original Kontsevich map \cite[Section 7]{Kontsevich} and its cyclic version \cite[Appendix]{WC}.

Before reading this section, we recommend the non-expert reader to read \cite[Appendix]{WC} that contains all the crucial arguments. We also recommend \cite[Section 4]{Dolgushev} for a detailed introduction to the Fedosov resolutions that we use. Let us nevertheless sketch the general argument.

We start by remarking that the entire construction of the $\Grav_\infty$ quasi-isomorphism $\Tpoly[u](\mathbb R^d) \to \Dpoly^\sigma(\mathbb R^d)$ still holds if we replace $\mathbb R^d$ by $\Rformal$, its formal completion at the origin.

 One considers $\mathcal T^{\text{formal}}_\text{poly}$ (resp. $\mathcal D^{\text{formal}}_\text{poly}$),  the vector bundle on $M$ of fiberwise formal multivector fields (resp. multidifferential operators) tangent to the fibers. As in the flat case, one can also consider their cyclic versions $\mathcal T^{\text{formal}}_\text{poly}[u]$ (with appropriate differential) and $(\mathcal D^{\text{formal}}_\text{poly})^\sigma$.
 
 We can then construct the vector bundles $\Omega(\mathcal T^{\text{formal}}_\text{poly}[u], M)$ of forms valued in $\mathcal T_\text{poly}[u]$ and $\Omega((\mathcal D^{\text{formal}}_\text{poly})^\sigma, M)$ of forms valued in $(\mathcal D^{\text{formal}}_\text{poly})^\sigma$ with appropriate differentials. 
 
The fibers of the bundles $\mathcal T^{\text{formal}}_\text{poly}[u]$ and $(\mathcal D_\text{poly}^{\text{formal}})^\sigma$ are isomorphic to $\Tpoly(\Rformal)$ and $\Dpoly(\Rformal)$, respectively. Therefore, the formal version of the formality map can be used to find a vector bundle $\Grav_\infty$ quasi-isomorphism

\begin{equation}\label{fiberwise morphism}
  U^f\colon \Omega(\mathcal T^{\text{formal}}_\text{poly}[u], M) \to \Omega((\mathcal D^{\text{formal}}_\text{poly})^\sigma, M). \footnote{Using the fact that the formality morphism is invariant by linear transformation of coordinates.}
\end{equation}

These two vector bundles can be related with $\Tpoly[u](M)$ and $\Dpoly^\sigma(M)$. In fact,  with an appropriate change of differential that comes from a choice of a flat connection, $\Omega(\mathcal T^{\text{formal}}_\text{poly}[u], M)$ becomes a resolution of $\Tpoly[u](M)$ and $\Omega((\mathcal D^{\text{formal}}_\text{poly})^\sigma, M)$ becomes a resolution of $\Dpoly^\sigma(M)$. Both changes of differential can be seen locally as a twist via a Maurer-Cartan element $B$ \footnote{This $B$ is the same one that one uses in the non-cyclic setting. The fact that $B$ is still a Maurer-Cartan element in $\Omega(\mathcal T^{\text{formal}}_\text{poly}[u], U)$ follows from it being divergence free \cite[Proposition 27]{WC}.} sitting inside $\Omega^1(\mathcal T^{\text{formal},1}_\text{poly}[u], U)$  or $\Omega^1((\mathcal D^{\text{formal}}_\text{poly})^{\sigma,1}, U)$. 

However, the linear part of $B$ (in the fiber coordinates) is not globally well defined. It follows that to show that the globalization of the $\Grav_\infty$ map is possible, it suffices to see that its construction is compatible with twisting by Maurer--Cartan elements in a way that is not using the linear part of $B$. \\

There are are three main components in the globalization procedure:

\begin{enumerate}
    \item The $\Grav_\infty$ formality morphism needs to be made compatible with twisting, 
    \item The $\mathfrak{s}^{-1}\Lie_\infty$ piece of the $\Grav_\infty$ map must send $B$ to itself,
    \item The twisting procedure must not use the linear part of $B$.
\end{enumerate}

We remark that the second condition is automatically satisfied since the $\mathfrak{s}^{-1}\Lie_\infty$ piece of the $\Grav_\infty$ map is precisely Willwacher's formality map which satisfies this property.

The first component is essentially done by operadic twisting together with the verification of a condition of native twistability at the level of $\Chains(\mathbb H_{\bullet,0})$. The third component consists of checking that after the twisting procedure, the obtained $\Grav_\infty$ morphism factors through graphs whose action does not use the linear part of $B$. As we will see later, this would occur  whenever there exist internal vertices with exactly one outgoing edge and at most one incoming edge (since more incoming edges would kill the linear part).

\subsection{The approach using operadic twisting}
Let us recall the formalism of operadic twisting, developed extensively in \cite{OpTwisting}. Most of it adapts in a straightforward manner to the operadic bimodule  setting, as explained in the Appendix of \cite{Campos}.
Let $\cP$ be an operad under $\mathfrak{s}^{-1}\Lie_\infty$. If one twists a $\cP$-algebra $A$ (in particular a $\mathfrak{s}^{-1}\Lie_\infty$-algebra) by a Maurer-Cartan element $\mu\in A$, the resulting twisted algebra $A^\mu$ is not an algebra over $\cP$ but rather over the twisted $\Tw \cP$. 

However, if $\cP$ is \textit{natively twistable}, i.e., there exists an operad morphism $\cP\to \Tw \cP$ such that $\cP\to \Tw \cP\to \cP$ is the identity, then $\cP$ still acts on $A$.

Recall that the action of $\Grav_\infty$ on $\Tpoly[u]$ can be expressed as a map 

\begin{equation}\label{eq: Grav action on Tpoly}
\Grav_\infty \to \tcc(\Chains(\FM_2)) \to \tcc(\Gra) \to \End_{\Tpoly[u]}
\end{equation}
inducing a similar action on $\Omega(\mathcal T^{\text{formal}}_\text{poly}[u],M)$. 
Unfortunately, the functor $\tcc$ does not behave well with respect to operadic twisting. For instance, given a map $\mathfrak{s}^{-1}\Lie\to \op P$, there is no natural map $\mathfrak{s}^{-1}\Lie\to \tcc(\op P)$. On the other hand, as the following lemma shows, we can circumvent this issue by considering the functor $\nhc$ instead.

\begin{lemma} Let $\mu\colon(\mathfrak{s}^{-1}\Lie,0,0) \to (\op{P},d,\Delta)$ be a morphism in $\mxdops$. (So in particular $\mu(l_2)\in ker(\Delta)$.)  Then there is a morphism $\hat{\mu}\colon \mathfrak{s}^{-1}\Lie \to\nhc(\op{P})$ for which
\begin{equation*}
    \nhc(\mathsf{Tw}^\mu(\op{P}))\hookrightarrow \mathsf{Tw}^{\hat{\mu}}(\nhc(\op{P}))
\end{equation*}
\end{lemma}
\begin{proof}  The morphism $\hat{\mu}\colon \mathfrak{s}^{-1}\Lie \to\nhc(\op{P})$ is given by $f(-)\otimes u^0$, which is a dg map since $f$ lands in the kernel of $d$ and of $\Delta$.

Now on the level of graded vector spaces we can include
\begin{equation*}
    (\ds\prod_{r\geq 0} \op{P}(n+r))\tensor k[u]   \hookrightarrow \ds\prod_{r\geq 0}( \op{P}(n+r)\tensor k[u])
\end{equation*}
as the subset of lists whose powers of $u$ match.  Here we view $\mathsf{Tw}(\op{P})$ as having a mixed complex structure via the product over $r$ of $\Delta_{n+r}\colon \op{P}(n+r)\to\op{P}(n+r)$.  The differential on the left hand side is $(d_\op{P}+ d^\mathsf{Tw}_\mu)+u\Delta$.  The differential on the right hand side is $(d_\op{P}+u\Delta)+ d^\mathsf{Tw}_{\hat{\mu}}$.  So since the inclusion takes $\{\mu(l_2),-\}\tensor u^0$ to $\{\mu(l_2)\tensor u^0,-\}$ it turns $d^\mathsf{Tw}_\mu$ into $d^\mathsf{Tw}_{\hat{\mu}}$, whence the claim. 
\end{proof}

We can then reexpress the action \eqref{eq: Grav action on Tpoly} as
$$\Grav_\infty \to \tcc(\Chains(\FM_2)) \to \tcc(\Gra)\to \nhc(\Gra) \to \End_{\Tpoly[u]}.$$

If we factor the map $\tcc(\Chains(\FM_2))\to \nhc(\Gra)$ through the canonical projection $\Tw \nhc(\Gra)\to \nhc(\Gra)$, we will obtain a $\Grav_\infty$ structure on  $\Omega(\mathcal T^{\text{formal}}_\text{poly}[u],M)^\mu$ for every Maurer--Cartan element $\mu$ given by the following maps

$$\Grav_\infty \to \tcc(\Chains(\FM_2)) \to  \Tw \nhc(\Gra) \to \End_{\Omega(\mathcal T^{\text{formal}}_\text{poly}[u])^\mu}.$$

In fact, looking at diagram \eqref{diagram} using operadic bimodule twisting\footnote{c.f. \cite[Appendix]{Campos} regarding twisting of operadic bimodules.}, we see that the same argument can be used to twist the $\Grav_\infty$ morphism, as long as we can find a factorization of the following form:
\begin{equation}
\xymatrix{ 
	\mathsf M_\circlearrowright \ar[d] & \circlearrowright & \ar[d] \Chains(\mathbb H_{\bullet,0})  & \circlearrowleft &  \tcc(\Chains(\FM_2))  \ar[d] \\ 
		\Tw \mathsf M_\circlearrowright \ar[d] & \circlearrowright & \ar[d]\Tw \left(\prod_n \Sigma^n\vKGra(\bullet, n)^{\mathbb{Z}_{n+1}}\right)^\mu  & \circlearrowleft &  \Tw \nhc(\Gra)  \ar[d] \\ 
	\mathsf M_\circlearrowright  & \circlearrowright &   \left(\prod_n \Sigma^n\vKGra(\bullet, n)^{\mathbb{Z}_{n+1}}\right)^\mu  & \circlearrowleft &  \nhc(\Gra) }
\end{equation}

In fact, due to the ill-definedness of the linear part of the Maurer--Cartan element $B$ that we consider, we must in fact factor the morphism through a smaller bimodule which we construct in the next section.

\subsection{Twisting of Graphs}

The construction of this section is essentially a formal adaptation of the globalization section in \cite{Campos}, so we will only sketch it and refer to loc.\ cit.\ for the missing proofs. We first need the following proposition whose proof is immediate.

 \begin{proposition}\label{prop:CC- of CSC}
	If $\op{P}=(\op{P}^1,\op{P}^2)$ is an operad of CSC type  and if the operad $\op{P}^1$ is a rotational operad, then $\nhc(\cP)=(\nhc(\cP^1),\cP^2)$ is still an operad of CSC type, with compositions given by
	\begin{equation*}
	p_2 \  \tilde{\circ}_l \  p_1u^k=\begin{cases} p_2 \circ_l p_1 &\text{ if }k=0  \\ 0 & \text{ if } k>0, \end{cases}
	\end{equation*}
	for $p_i\in \op P^i$.
	Moreover, the map from corollary \ref{wecor} induces a morphism of CSC type operads $\tcc(\op P) \to \nhc(\op P)$.
\end{proposition}

The $\mathsf{M}_\circlearrowright-\nhc(\Gra)$-bimodule $\left(\prod_n \Sigma^n\vKGra(\bullet, n)^{\mathbb{Z}_{n+1}}\right)^\mu$ constructed in section \ref{sec:repon col_op} (together with Proposition \ref{prop:CC- of CSC}) can be twisted to obtain the $\Tw \mathsf{M}_\circlearrowright-\Tw\nhc(\Gra)$-bimodule $\Tw \left(\prod_n \Sigma^n\vKGra(\bullet, n)^{\mathbb{Z}_{n+1}}\right)^\mu$. Notice that $\mathsf{M}_\circlearrowright$ arises itself from operadic twisting and we can therefore restrict the left action of $\Tw \mathsf{M}_\circlearrowright$ to $\mathsf{M}_\circlearrowright$ using the map $\mathsf{M}_\circlearrowright \to \Tw \mathsf{M}_\circlearrowright$.

Recall from section \ref{sec: formality of little discs} the operad $\Graphs$, defined as the suboperad of $\Tw \Gra$ spanned by graphs such that all internal vertices have $\geq 3$ valence and every connected component contains at least an external vertex.

We can restrict the bimodule right action to $\tcc(\Graphs)$ via the chain of inclusions
$ \nhc(\Graphs) \subset \nhc(\Tw \Gra) \subset \Tw \nhc(\Gra).$

\begin{propdef}
    The $\mathsf{M}_\circlearrowright-\nhc(\Graphs)$ bimodule $\Tw \left(\prod_n \Sigma^n\vKGra(\bullet, n)^{\mathbb{Z}_{n+1}}\right)^\mu $ has a sub-quotient denoted by $\vKGraphs^{\sigma}$ constructed in the following way: 
    
    We first consider the quotient $Q$ of $\Tw \left(\prod_n \Sigma^n\vKGra(\bullet, n)^{\mathbb{Z}_{n+1}}\right)^\mu$ by the subspace consisting of graphs with tadpoles or powers of $v$ on type I internal vertices and then the subspace of $Q$ spanned by the graphs with the following properties:
    \begin{enumerate}
\item There is at least one type I external vertex,
\item There are no 0-valent type I internal vertices
\item There are no 1-valent type I internal vertices with an outgoing edge,
\item There are no 2-valent type I internal vertices with one incoming and one outgoing edge.
\end{enumerate}
\end{propdef}
\begin{proof}

This result is essentially \cite[Def./Prop. 24]{Campos}, where it was done for $\BVKGraphs$, since $\BVKGra$ can be interpreted as the quotient of $\vKGra$ by graphs containing non-zero powers of $v$. We sketch the proof pointing out the adaptations to our case.

 The right $\nhc(\Graphs)$ action cannot destroy tadpoles on internal vertices hence it descends to $Q$. $\vKGraphs^\sigma$ is clearly stable by the right action. 
 
 To verify the stability by the left action and by the differential one uses two properties of the Maurer--Cartan element $m$ (the image of the generators of $\sLie_\infty^{\text{bimod}}$) by which we twist:

\begin{enumerate}[(a)]
\item The only graph in $m$ containing a $1$-valent type I internal vertex is the 2 vertex graph 
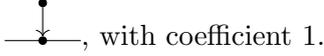
\begin{tikzpicture}[scale=1]

\node (i) at (1,0.5) [int] {};
\node (j) at (1,0) [int] {};

\draw [->] (i)--(j);
\draw  (0.5,0)--(1.5,0);
\end{tikzpicture}, with coefficient $1$.
\item There are no graphs with vertices like the ones in property \textit{(4)}.
\end{enumerate}

The proof of these properties is the same as for the original Kontsevich vanishing lemmas. 
 
Using these properties it is a straightforward (but lengthy) combinatorial verification that non-cyclic invariant graphs $\vKGraphs \supset \vKGraphs^\sigma$ are preserved by the left $\mathsf M$ action. It follows that the cyclic invariant $\vKGraphs^\sigma$ are preserved by the $\mathsf M_\circlearrowleft$.
 
 Similarly, one can check that $\vKGraphs$ are stable by the differential and to see that the cyclic invariant $\vKGraphs^\sigma$ are preserved by the differential it is enough to notice that the image of the generators of $\sLie_\infty^{\text{bimod}} $ is cyclic invariant itself.
\end{proof}

\subsection{Factorization of the bimodule morphism}

To conclude the globalization procedure it is enough to construct the first bimodule morphism of the following diagram:
\begin{equation}
\xymatrix{ 
	\mathsf M_\circlearrowright \ar[d]^{\text{id}} & \circlearrowright & \ar[d]^{f} \Chains(\mathbb H_{\bullet,0})  & \circlearrowleft &  \tcc(\Chains(\FM_2))  \ar[d]^{g} \\ 
	\mathsf M_\circlearrowright \ar[d] & \circlearrowright & \ar[d] \vKGraphs^\sigma  & \circlearrowleft &  \nhc(\Graphs)  \ar[d] \\ 
		\End_{\Omega((\mathcal D^{\text{formal}}_\text{poly})^\sigma, M)^B} & \circlearrowright & \End^{\Omega(\mathcal T^{\text{formal}}_\text{poly}[u], M)^B}_{\Omega((\mathcal D^{\text{formal}}_\text{poly})^\sigma, M)^B}& \circlearrowleft &  \End_{\Omega(\mathcal T^{\text{formal}}_\text{poly}[u], M)^B} }
\end{equation}

The map $g$ is defined to be the composition
$$\tcc(\Chains(\FM_2))\to \nhc (\Chains(\FM_2)) \stackrel{\nhc(\pi^{-1})}{\to} \nhc (\Tw \Chains(\FM_2))\to \nhc(\Tw \Gra).$$

Here we consider the maps
$$\pi_n^{-1} = \prod_k \pi_{n,k}^{-1}\colon \Chains(\mathsf{FM}_2(n)) \to \Tw \Chains(\FM_2(n)) = \prod_k \Sigma^{2k} \Chains(\FM_2(n+k))^{\mathbb S_k},$$ obtained as the strongly continuous chain associated to the SA bundle corresponding to the map $\pi_{n,k}\colon \FM_2(n+k) \to \FM_2(n)$ that forgets the last $k$ points. Informally, the map $\pi_{n,k}^{-1}$ is obtained by creating $k$ points that move freely. 

The maps $\pi_n^{-1}$ are clearly compatible with the cyclic action and therefore induce the desired 
$$ \nhc(\Chains(\FM_2)) \to \nhc(\Tw \Chains(\FM_2)).$$ 

Notice that fact that the composition $\Chains(\FM_2)\to \Tw \Chains(\FM_2) \to \Tw \Gra $ actually lands inside $\Graphs$ uses Kontsevich's vanishing lemmas \cite{Kontsevich}.\\

The map $f$ is given by the composition: $\Chains(\mathbb H_{m,0}) \stackrel{\pi^{-1}}{\to}\prod_{k} \Sigma^{2k}\Chains(\mathbb H_{m+k,0}) \to$
\begin{equation*}
\left(\prod_{n,k\geq 0} \Sigma^{n+2k}\Chains(\mathbb H_{m+k,n})^{\mathbb Z_{n+1}}\right)^\mu\to \Tw \left(\prod_n \Sigma^n\vKGra(m,n)^{\mathbb Z_{n+1}}\right)^\mu.
\end{equation*}

Here, $\pi^{-1}$ is defined, as above, as the strongly continuous chain associated to the projection $\mathbb H_{m+k,n} \to \mathbb H_{m,n}$.

To finish the globalization argument, one needs to check the following two properties:

\begin{tabular}{l}
     \textit{(i)} $f$ is a map of bimodules, \\
     \textit{(ii)}  $f$ lands in $\vKGraphs^\sigma(m)$ seen as a subquotient of $\Tw \left(\prod_n \Sigma^n\vKGra(m,n)^{\mathbb Z_{n+1}}\right)^\mu$.\\
\end{tabular}

\subsubsection{Proof of (i)}

We start by noticing that the compatibility with the left $\mathsf{M}_\circlearrowleft$ is immediate. As for the right action, notice that $f$ as a right module map can be decomposed as 
\begin{equation}
\xymatrix{ 
\ar[d] \Chains(\mathbb H_{\bullet,0})  & \circlearrowleft &  \tcc(\Chains(\FM_2))  \ar[d]^{g'} \\ 
\ar[d]  \left(\prod_{n,k\geq 0} \Sigma^{n+2k}\Chains(\mathbb H_{\bullet+k,n})^{\mathbb Z_{n+1}}\right)^\mu   & \circlearrowleft &  \nhc(\Tw \Chains(\FM_2))  \ar[d] \\  
 \Tw \left(\prod_n \Sigma^n\vKGra(m,n)^{\mathbb Z_{n+1}}\right)^\mu  & \circlearrowleft &  \nhc(\Graphs).  \\ 
 }
\end{equation}
The upper map is easily checked to be a morphism of right modules. However, due to remark \ref{rem: not CSC map} the  bottom map is not a morphism of right modules. However, it is so if we restrict it to the image of $g'$, essentially by Proposition \ref{prop: need CCtheta}. This guarantees that $f$ itself is a morphism of right modules.

The compatibility of $f$ with the differential follows from the same arguments as the functoriality of bimodule twisting.

\subsubsection{Proof of (ii)} 

One has to show that every graph not satisfying at least one of properties \textit{(1), (2), (3)} or \textit{(4)} appears in the image of $f$ with coefficient zero. This is clear for the first property. 

As for property \textit{(2)}, if a graph contains an isolated type I internal vertex, its coefficient will involve the integration of a $0$-form over a two dimensional space, which is zero.

Similarly, if a graph contains a $1$-valent internal vertex, its coefficient will involve an integral of a $1$-form over a two dimensional space and is therefore $0$.

Finally, if a graph has an internal vertex $i$ connected to vertices $a$ and $b$ as in property \textit{(4)}, in the computation of its coefficient we find the factor 

$$\int_{X_{z_a,z_b}} d\phi^{ai}d\phi^{ib}  $$ 
 where $X_{z_a,z_b}$ is the space of configurations in which the points labeled by $a$ and $b$ are in positions $z_a$ and $z_b$, and the point labeled by $i$ moves freely. Here the notation assumes that both $a$ and $b$ are type I vertices but the argument also holds if they are type II vertices.
 
By Stokes' theorem for SA bundles, we have
 
$$d \underbrace{\int_{Y_{z_a,z_b}} d\phi^{ai}d\phi^{ij}d\phi^{jb}}_{0} = \int_{Y_{z_a,z_b}} \underbrace{d(d\phi^{ai}d\phi^{ij}d\phi^{jb})}_{0} \pm \int_{\partial Y_{z_a,z_b}} d\phi^{ai}d\phi^{ij}d\phi^{jb},$$ 
 where $Y_{z_a,z_b}$ is the configuration space of four points ($i,j,a$ and $b$) where $a$ and $b$ are fixed at $z_a$ and $z_b$ and the points labeled by $i$ and $j$ are free. The integral on the left hand side vanishes by degree reasons. The boundary terms on the right hand side vanish except on the following cases:
 
 \begin{itemize}
 \item The boundary stratum in which $a$ and $i$ are infinitely close,
 
 \item The boundary stratum in which $i$ and $j$ are infinitely close,
 
 \item The boundary stratum in which $j$ and $b$ are infinitely close.
 \end{itemize}
In each of these cases, the result is an integral of the form $\int_{X_{z_a,z_b}} d\phi^{ai}d\phi^{ib}  $, therefore it is zero.


\bibliography{Wgravbib}

\begin{thebibliography}{HLTV11}

\bibitem[Arn69]{Arnold}
V.~I. Arnol'd.
\newblock The cohomology ring of the group of dyed braids.
\newblock {\em Mat. Zametki}, 5:227--231, 1969.

\bibitem[Cam17]{Campos}
Ricardo Campos.
\newblock B{V} formality.
\newblock {\em Adv. Math.}, 306:807--851, 2017.

\bibitem[CW16]{torsors}
Ricardo Campos and Thomas Willwacher.
\newblock Operadic torsors.
\newblock {\em J. Algebra}, 458:71--86, 2016.

\bibitem[Dol06]{Dolgushev}
Vasiliy Dolgushev.
\newblock A formality theorem for {H}ochschild chains.
\newblock {\em Adv. Math.}, 200(1):51--101, 2006.

\bibitem[DW15]{OpTwisting}
Vasily Dolgushev and Thomas Willwacher.
\newblock Operadic twisting---with an application to {D}eligne's conjecture.
\newblock {\em J. Pure Appl. Algebra}, 219(5):1349--1428, 2015.

\bibitem[FM94]{FM}
William Fulton and Robert MacPherson.
\newblock A compactification of configuration spaces.
\newblock {\em Ann. of Math. (2)}, 139(1):183--225, 1994.

\bibitem[Fre17]{FresseBook}
Benoit Fresse.
\newblock Homotopy of operads and grothendieck--teichm\"uller groups.
\newblock {\em preprint}, 1, 2017.

\bibitem[Get94]{Geteq}
E.~Getzler.
\newblock Two-dimensional topological gravity and equivariant cohomology.
\newblock {\em Comm. Math. Phys.}, 163(3):473--489, 1994.

\bibitem[GJ94]{Getzler-Jones}
E.~Getzler and J.D.S. Jones.
\newblock Operads, homotopy algebra and iterated integrals for double loop
  spaces.
\newblock {\em arxiv.org/abs/hep-th/9403055}, 1994.

\bibitem[GK95]{GeK1}
E.~Getzler and M.~M. Kapranov.
\newblock Cyclic operads and cyclic homology.
\newblock In {\em Geometry, topology, \& physics}, Conf. Proc. Lecture Notes
  Geom. Topology, IV, pages 167--201. Int. Press, Cambridge, MA, 1995.

\bibitem[GS10]{GS10}
Jeffrey Giansiracusa and Paolo Salvatore.
\newblock Formality of the framed little 2-discs operad and semidirect
  products.
\newblock In {\em Homotopy theory of function spaces and related topics},
  volume 519 of {\em Contemp. Math.}, pages 115--121. Amer. Math. Soc.,
  Providence, RI, 2010.

\bibitem[GV95]{GV}
Murray Gerstenhaber and Alexander~A. Voronov.
\newblock Homotopy {$G$}-algebras and moduli space operad.
\newblock {\em Internat. Math. Res. Notices}, (3):141--153, 1995.

\bibitem[Hin03]{Hinich}
Vladimir Hinich.
\newblock Tamarkin's proof of {K}ontsevich formality theorem.
\newblock {\em Forum Math.}, 15(4):591--614, 2003.

\bibitem[HLTV11]{HLTV}
Robert Hardt, Pascal Lambrechts, Victor Turchin, and Ismar Voli\'c.
\newblock Real homotopy theory of semi-algebraic sets.
\newblock {\em Algebr. Geom. Topol.}, 11(5):2477--2545, 2011.

\bibitem[Kau05]{KCacti}
Ralph~M. Kaufmann.
\newblock On several varieties of cacti and their relations.
\newblock {\em Algebr. Geom. Topol.}, 5:237--300 (electronic), 2005.

\bibitem[Kau08]{CycDel}
Ralph~M. Kaufmann.
\newblock A proof of a cyclic version of {D}eligne's conjecture via cacti.
\newblock {\em Math. Res. Lett.}, 15(5):901--921, 2008.

\bibitem[Kon99]{Kont99}
Maxim Kontsevich.
\newblock Operads and motives in deformation quantization.
\newblock {\em Lett. Math. Phys.}, 48(1):35--72, 1999.
\newblock Mosh\'e Flato (1937--1998).

\bibitem[Kon03]{Kontsevich}
Maxim Kontsevich.
\newblock Deformation quantization of {P}oisson manifolds.
\newblock {\em Lett. Math. Phys.}, 66(3):157--216, 2003.

\bibitem[KS00]{KS}
Maxim Kontsevich and Yan Soibelman.
\newblock Deformations of algebras over operads and the {D}eligne conjecture.
\newblock In {\em Conf\'erence {M}osh\'e {F}lato 1999, {V}ol. {I} ({D}ijon)},
  volume~21 of {\em Math. Phys. Stud.}, pages 255--307. Kluwer Acad. Publ.,
  Dordrecht, 2000.

\bibitem[LV14]{discs}
Pascal Lambrechts and Ismar Voli\'c.
\newblock Formality of the little {$N$}-disks operad.
\newblock {\em Mem. Amer. Math. Soc.}, 230(1079):viii+116, 2014.

\bibitem[Sho99]{Shoikhet}
Boris Shoikhet.
\newblock On the cyclic formality conjecture.
\newblock {\em arxiv.org/abs/math/9903183}, preprint, 1999.

\bibitem[Tam98]{Tamarkin}
D.~Tamarkin.
\newblock Another proof of m. kontsevich formality theorem.
\newblock {\em arXiv:math/9803025}, 1998.

\bibitem[Vor99]{Voronov}
Alexander~A. Voronov.
\newblock The {S}wiss-cheese operad.
\newblock In {\em Homotopy invariant algebraic structures ({B}altimore, {MD},
  1998)}, volume 239 of {\em Contemp. Math.}, pages 365--373. Amer. Math. Soc.,
  Providence, RI, 1999.

\bibitem[Vor05]{Vor}
Alexander~A. Voronov.
\newblock Notes on universal algebra.
\newblock In {\em Graphs and patterns in mathematics and theoretical physics},
  volume~73 of {\em Proc. Sympos. Pure Math.}, pages 81--103. Amer. Math. Soc.,
  Providence, RI, 2005.

\bibitem[War16]{Ward2}
Benjamin~C. Ward.
\newblock Maurer {C}artan elements and cyclic operads.
\newblock {\em J. Noncomm. Geom.}, 10(4):1403--1464, 2016.

\bibitem[WC12]{WC}
Thomas Willwacher and Damien Calaque.
\newblock Formality of cyclic cochains.
\newblock {\em Adv. Math.}, 231(2):624--650, 2012.

\bibitem[Wes08]{Westerland}
Craig Westerland.
\newblock Equivariant operads, string topology, and {T}ate cohomology.
\newblock {\em Math. Ann.}, 340(1):97--142, 2008.

\bibitem[Wil16]{homotopybracesformality}
Thomas Willwacher.
\newblock The homotopy braces formality morphism.
\newblock {\em Duke Math. J.}, 165(10):1815--1964, 2016.

\end{thebibliography}
\bibliographystyle{alpha}

\end{document}